\RequirePackage{pdf14}
\documentclass[notitlepage,11pt,reqno]{amsart}
\usepackage[foot]{amsaddr}
\usepackage{amssymb,nicefrac,bm,upgreek,mathtools,verbatim,enumerate}
\usepackage[final]{hyperref}
\usepackage[mathscr]{eucal}
\usepackage{dsfont}
\usepackage[normalem]{ulem}
\usepackage{amsopn}

\usepackage[utf8]{inputenc}%for appendix
\usepackage{apptools}%for appendix
\usepackage{scalerel}
\usepackage{tikz}
\usetikzlibrary{svg.path}

\definecolor{orcidlogocol}{HTML}{A6CE39}
\tikzset{orcidlogo/.pic={\fill[orcidlogocol]
svg{M256,128c0,70.7-57.3,128-128,128C57.3,256,0,198.7,0,128C0,57.3,57.3,0,128,
    0C198.7,0,256,57.3,256,128z};
\fill[white] svg{M86.3,186.2H70.9V79.1h15.4v48.4V186.2z}
svg{M108.9,79.1h41.6c39.6,0,57,28.3,57,53.6c0,27.5-21.5,53.6-56.8,
53.6h-41.8V79.1zM124.3,172.4h24.5c34.9,0,42.9-26.5,
42.9-39.7c0-21.5-13.7-39.7-43.7-39.7h-23.7V172.4z}
svg{M88.7,56.8c0,5.5-4.5,10.1-10.1,10.1c-5.6,0-10.1-4.6-10.1-10.1c0-5.6,
4.5-10.1,10.1-10.1C84.2,46.7,88.7,51.3,88.7,56.8z};}}

\newcommand\orcidicon[1]{\href{https://orcid.org/#1}{\mbox{\scalerel*{
\begin{tikzpicture}[yscale=-1,transform shape]
\pic{orcidlogo};
\end{tikzpicture}}{|}}}}

\usepackage[margin=1in]{geometry}
\newcommand{\stkout}[1]{\ifmmode\text{\sout{\ensuremath{#1}}}\else\sout{#1}\fi}
\newtheorem{lemma}{Lemma}[section]
\newtheorem{theorem}{Theorem}[section]
\newtheorem{proposition}{Proposition}[section]

\theoremstyle{definition}
\newtheorem{definition}{Definition}[section]
\newtheorem{assumption}{Assumption}[section]

\newtheorem{hypothesis}{Hypothesis}[section]
\newtheorem{example}{Example}[section]

\theoremstyle{remark}
\newtheorem{remark}{Remark}[section]
\numberwithin{theorem}{section}
\numberwithin{equation}{section}

\hypersetup{
 colorlinks=true,
 citecolor=mblue,
 linkcolor=mblue,
 urlcolor = blue,
 anchorcolor = blue,
 frenchlinks=false,
 pdfborder={0 0 0},
 naturalnames=false,
 hypertexnames=false,
 breaklinks}
\usepackage[capitalize,nameinlink]{cleveref}

\usepackage[abbrev,msc-links,nobysame,citation-order,backrefs]{amsrefs}
\crefname{section}{Section}{Sections}
\crefname{subsection}{Section}{Sections}
\crefname{condition}{Condition}{Conditions}
\crefname{hypothesis}{Hypothesis}{Conditions}
\crefname{assumption}{Assumption}{Assumptions}
\crefname{lemma}{Lemma}{Lemmas}
\crefname{fact}{Fact}{Facts}

\Crefname{figure}{Figure}{Figures}

\crefformat{equation}{\textup{#2(#1)#3}}
\crefrangeformat{equation}{\textup{#3(#1)#4--#5(#2)#6}}
\crefmultiformat{equation}{\textup{#2(#1)#3}}{ and \textup{#2(#1)#3}}
{, \textup{#2(#1)#3}}{, and \textup{#2(#1)#3}}
\crefrangemultiformat{equation}{\textup{#3(#1)#4--#5(#2)#6}}%
{ and \textup{#3(#1)#4--#5(#2)#6}}{, \textup{#3(#1)#4--#5(#2)#6}}%
{, and \textup{#3(#1)#4--#5(#2)#6}}

\Crefformat{equation}{#2Equation~\textup{(#1)}#3}
\Crefrangeformat{equation}{Equations~\textup{#3(#1)#4--#5(#2)#6}}
\Crefmultiformat{equation}{Equations~\textup{#2(#1)#3}}{ and \textup{#2(#1)#3}}
{, \textup{#2(#1)#3}}{, and \textup{#2(#1)#3}}
\Crefrangemultiformat{equation}{Equations~\textup{#3(#1)#4--#5(#2)#6}}%
{ and \textup{#3(#1)#4--#5(#2)#6}}{, \textup{#3(#1)#4--#5(#2)#6}}%
{, and \textup{#3(#1)#4--#5(#2)#6}}

\crefdefaultlabelformat{#2\textup{#1}#3}
\newcommand{\ttup}[1]{\textup{(}#1\textup{)}}

\newcommand{\vertiii}[1]{{\left\vert\kern-0.25ex\left\vert\kern-0.25ex\left\vert #1
   \right\vert\kern-0.25ex\right\vert\kern-0.25ex\right\vert}}
\newcommand{\lamstr}{\lambda^{\mspace{-2mu}*}}% lambda with corrected 'star'
\newcommand{\cA}{{\mathcal{A}}}  % generator
\newcommand{\bcA}{\boldsymbol{\mathcal{A}}}  % generator
\newcommand{\sB}{{\mathscr{B}}}  % Ball
\newcommand{\cC}{{C}}   % Continuous functions
\newcommand{\sE}{{\mathscr{E}}}
\newcommand{\sH}{{\mathscr{H}}}  % class of invariant measures
\newcommand{\sK}{{\mathscr{K}}}  % compact set
\newcommand{\Lg}{{\mathcal{L}}}    % generator
\newcommand{\bLg}{\boldsymbol{\mathcal{L}}}    % generator
\newcommand{\Lp}{{L}}            % Lp
\newcommand{\cS}{{\mathcal{S}}}
\newcommand{\Lyap}{{\mathcal{V}}}  % Lyapunov
\newcommand{\cX}{{\mathcal{X}}}

\newcommand{\RR}{\mathds{R}}
\newcommand{\NN}{\mathds{N}}

\newcommand{\Rd}{{\mathds{R}^{d}}}
\DeclareMathOperator{\Exp}{\mathbb{E}}
\DeclareMathOperator{\Prob}{\mathbb{P}}
\newcommand{\D}{\mathrm{d}}
\newcommand{\E}{\mathrm{e}}
\newcommand{\Ind}{\mathds{1}}   % indicator function
\newcommand{\Vt}{\Delta}  %intervals of R with respecto rates
\newcommand{\Sob}{{\mathscr W}}    % Sobolev Space
\newcommand{\Sobl}{{\mathscr W}_{\text{loc}}} % Sobolev Space(local)

\newcommand{\Id}{\bm{I}}
\newcommand{\df}{\coloneqq}
\newcommand{\transp}{^{\mathsf{T}}}
\newcommand{\bPsi}{{\boldsymbol\Psi}}
\newcommand{\bPhi}{{\boldsymbol\Phi}}

\DeclareMathOperator*{\diag}{diag}
\DeclareMathOperator*{\trace}{Tr}
\DeclareMathOperator*{\dist}{dist}

\DeclareMathOperator*{\supp}{support}

\DeclareMathOperator{\sign}{sign}

\newcommand{\sorder}{{\mathfrak{o}}}
\newcommand{\grad}{\nabla}
\newcommand{\uuptau}{{\Breve\uptau}}

\newcommand{\abs}[1]{\lvert#1\rvert}
\newcommand{\norm}[1]{\lVert#1\rVert}
\newcommand{\babs}[1]{\bigl\lvert#1\bigr\rvert}

\newcommand{\bnorm}[1]{\bigl\lVert#1\bigr\rVert}

\usepackage{color}
\definecolor{dmagenta}{rgb}{.4,.1,.5}
\definecolor{dblue}{rgb}{.0,.0,.4}
\definecolor{mblue}{rgb}{.0,.0,.7}
\definecolor{ddblue}{rgb}{.0,.0,.4}
\definecolor{dred}{rgb}{.7,.0,.0}
\definecolor{dgreen}{rgb}{.0,.5,.0}
\definecolor{Eeom}{rgb}{.0,.0,.5}

\newcommand{\ttl}{\Large On the monotonicity property of the generalized
eigenvalue\\[5pt] for weakly-coupled cooperative elliptic systems}
\begin{document}
\title[
Eigenvalues of weakly-coupled cooperative elliptic systems]
{\ttl}

\author[Ari Arapostathis]{Ari Arapostathis$^{\dag}$\protect\orcidicon{0000-0003-2207-357X}}
\address{$^{\dag}$Department of ECE,
The University of Texas at Austin,
EER~7.824, Austin, TX~~78712}
\email{ari@utexas.edu}%\,\protect\orcidicon{0000-0003-2207-357X}}

\author[Anup Biswas]{Anup Biswas$^\ddag$\protect\orcidicon{0000-0003-4796-2752}}
\address{$^\ddag$Department of Mathematics,
Indian Institute of Science Education and Research,
Dr.\ Homi Bhabha Road, Pune 411008, India}
\email{$\lbrace$anup,somnath$\rbrace$@iiserpune.ac.in}
%\protect\orcidicon{0000-0003-4796-2752}\,\protect\orcidicon{0000-0002-1470-8240}}

%\author[Anindya Goswami]{Anindya Goswami$^\ddag$}
%\email{anindya@iiserpune.ac.in}

\author[Somnath Pradhan]{Somnath Pradhan$^\ddag$\protect\orcidicon{0000-0002-1470-8240}}
%\email{somnath@iiserpune.ac.in}

%%%%%%%%%%%%%%%%%%%%%%%%%%%%%%%%%%%%%%%%%%%%%%%%%%%%%%%%%%%%%%%%%%%%%%%%%%%%%%%%%%%%
\begin{abstract}
We consider general linear non-degenerate weakly-coupled cooperative elliptic systems and study
certain monotonicity properties of the generalized principal eigenvalue in
$\Rd$ with respect to the potential. It is shown that monotonicity on the right is
equivalent to the recurrence property of the twisted operator which is, in turn,
equivalent to the minimal growth property at infinity of the principal eigenfunctions.
The strict monotonicity property of the principal eigenvalue is shown to be equivalent
with the
exponential stability of the twisted operators. An equivalence between the
monotonicity property on the right and the stochastic representation of the
principal eigenfunction 
is also established.
\end{abstract}
%%%%%%%%%%%%%%%%%%%%%%%%%%%%%%%%%%%%%%%%%%%%%%%%%%%%%%%%%%%%%%%%%%%%%%%%%%%%%%%%%%%%


%%%%%%%%%%%%%%%%%%%%%%%%%%%%%%%%%%%%%%%%%%%%%%%%%%%%%%%%%%%%%%%%%%%%%%%%%%%%%%%%%%%%
\maketitle

%%%%%%%%%%%%%%%%%%%%%%%%%%%%%%%%%%%%%%%%%%%%%%%%%%%%%%%%%%%%%%%%%%%%%%%%%%%%%%%%%%%%

\section{Introduction}

Regime switching diffusions are heavily used for modelling purposes in applied subjects like
mathematical finance \cite{Merton,BDY09,DKR94,ZY03}, wireless communications \cite{YKI04},
production planning \cite{SZ94}, predictive modelling \cite{GHM10,MH12}. See also the 
introduction
of the book by Yin and Zhu \cite{YZ10} for further motivation in studying regime switching 
diffusions.
Eigenvalue problems for weakly-coupled systems have also received a lot of attention. 
Most of the
existing works in this direction are concerned with
the maximum principle and Dirichlet 
principal eigenvalue
problems in bounded domains, see for instance,
Amann \cite{Amann04}, Birindelli et al.\ \cite{BMS99}, Cantrell and Schmitt 
\cite{Cantrell86}, Cantrell \cite{Cantrell88}, Hess \cite{Hess83}, Sweers \cite{Sweer}.
In this article we
consider the eigenvalue problem in the whole space $\Rd$ and study
its monotonicity properties with respect to 
the potential $\bm{c}$, and provide some sharp characterizations.
Our interest in these problems stems from its applications in risk-sensitive
control problems \cite{ABS19,FM95}. Apart from this, eigenvalue problems are also important 
in understanding
the large deviations behavior \cite{DV76,DVIII,Kaise-06} and
the Fisher-KPP type phenomenon 
\cite{Girardin}. In particular,
given a potential function $\bm{c}$ we consider the 
\textit{exponential-to-integration} 
(or risk-sensitive cost) function
given by
$$\sE(x,i) \,\df\, \limsup_{T\to\infty}\, 
\frac{1}{T}\,\log\Exp_{x,i}\left[\E^{\int_0^T \bm{c}(X_t, S_t)\, \D{t}}\right],$$
where $(X,S)$ represents the regime switching diffusion.
Such functionals are the main object in the study of risk-sensitive controls and large 
deviations phenomena.
It is often important to know under what circumstances we can have $\sE(x,i)=\lamstr$ 
where $\lamstr$ is the
principal eigenvalue of $\bcA=\bLg+\bm{c}$ in $\Rd$ and
$\bLg$ is the extended generator of $(X,S)$. As shown in \cite[Example~3.1]{ABS19} this 
equality does not hold in
general and the concepts of monotonicity (see \cref{D1.3} below)
of the principal eigenvalues were introduced
in \cite{ABS19}
to provide a sufficient condition for this equality to hold. It turns out that the 
concept of monotonicity
is also linked to the criticality of eigenfunctions used in  potential theory (see \cite{ABG19}
and references therein). In this article
we extend the study to weakly-coupled cooperative systems.
The results are
closely related to the works of Ichihara \cite{Ichihara-11,Ichihara-15} where
the eigenvalues of ergodic Hamilton-Jacobi equations are characterized through the 
recurrence/transience
behavior of the diffusion process governed by optimal feedback control.
An important aspect of this paper, is that most of the results are
obtained by analytical methods, instead using probabilistic arguments.
This is made possible by abstracting the notions of regularity, recurrence,
and geometric ergodicity of a diffusion to analogous notions for an operator
(see \cref{D1.2}).
This article
owes much to the work of  Berestycki and Rossi \cite{Berestycki-15} who recently study 
eigenvalue
problems for scalar elliptic equations in unbounded domains and its relation to maximum 
principles. It is also
possible to develop an analogous theory for systems but we do not pursue this direction in 
this article.

%%%%%%%%%%%%%%%%%%%%%%%%%%%%%%%%%%%%%%%%%%%%%%%%%%%%%%%%%%%%%%%%%%%%%%%%%%%%%%%%%%%%%
\subsection{The model and main results}
Let $\cS\df \{1,2,\dots, N\}$ be a discrete set.
In this paper we consider the generalized eigenvalue problem for weakly-coupled
elliptic systems $\bcA$ on $\Rd\times\cS$ taking the form
\begin{equation}\label{E1.1}
(\bcA\bm{f})_k(x) \,=\, \trace\bigl(a_k(x)\grad^2 f_{k}(x)\bigr)
+ b_k(x)\cdot \grad f_{k}(x) + \sum_{j\in\cS} c_{kj}(x)f_{j}(x)\,,
\quad k\in\{1,2,\dots, N\}\,,
\end{equation}
for $\bm{f} = (f_i)_{i\in\cS}$, with $f_i\in C^2(\Rd)$\,.
Here, $\nabla^2$ denotes the Hessian,
the coefficients $(a_k)_{k\in\cS}$ are continuous, positive definite symmetric matrices.
The coefficients $b_k\colon\Rd\to\Rd$, $k\in\cS$ and
$c_{ij}\colon\Rd\to\RR$, $i,j\in\cS$, are locally bounded and Borel measurable.
We consider cooperative systems, that is, $\bcA$ as above with
$c_{ij}\ge0$ for $i\ne j$. This property is also known as quasi-monotonicity property.
We also impose the following non-degeneracy condition throughout this article:
for each $R>0$, it holds that
\begin{equation*}
\sum_{i,j=1}^{d} a_k^{ij}(x)\zeta_{i}\zeta_{j}
\,\ge\,C^{-1}_{R} \abs{\zeta}^{2} \qquad\forall\, k\in\cS\,, x\in B_{R}\,,
\end{equation*}
and for all $\zeta=(\zeta_{1},\dotsc,\zeta_{d})\transp\in\RR^{d}$.
We prefer to write \cref{E1.1} in a form that has a probabilistic interpretation.
Let $c_k\df \sum_{j\in\cS} c_{kj}$, and the matrix $\bm M= [m_{ij}]_{i,j\in\cS}$
be defined by
\begin{equation*}
m_{ij} \,\df\, c_{ij}\text{\ \ for\ } i\ne j\,,
\quad\text{and\ } m_{ii} \,\df\, -\sum_{j\in\cS,\,j\ne i} c_{ij}\,.
\end{equation*}
Then, for each $x\in\Rd$, $\bm M(x)$ is a stochastic rate matrix of a
finite state Markov process.
We can write \cref{E1.1} in vector form as
\begin{equation*}
\bcA \bm{f} (x) \,=\, \bLg \bm f(x) + \bm{c}\bm{f}(x)\,,\quad x\in\Rd\,,
\end{equation*}
with
\begin{equation}\label{E-bLg}
\bigl(\bLg \bm{f}\bigr)_k (x) \,\df\,
\trace\bigl(a_k(x)\grad^2 f_{k}(x)\bigr)
+ b_k(x)\cdot \grad f_{k}(x) + \sum_{j\in\cS} m_{kj}(x)f_{j}(x)\,,
\quad (x,k)\in\Rd\times\cS\,,
\end{equation}
and
$$(\bm{c}\bm{f})_k \,=\, c_k f_k\,.$$

The operator $\bLg$ in \cref{E-bLg} is the extended generator of
a regime switching diffusion in $\Rd$ (see description in \cref{S-prob}).
One could go one more step further in this representation.
Define the collection $(L_k)_{k\in\cS}$ of elliptic operators on $\Rd$ by
\begin{equation}\label{E-Lk}
L_k g (x) \,\df\, \trace\bigl(a_k(x)\grad^2 g(x)\bigr)
+ b_k(x)\cdot \grad g(x)\,,\quad (x,k)\in\Rd\times\cS\,,\quad g\in C^2(\Rd)\,.
\end{equation}
Then, if we let $\bm{L}\df\diag(L_1,\dotsc,L_N)$, the operator $\bLg$ can be written
in vector notation as
\begin{equation*}
\bLg \bm{f} (x) \,=\, \bm{L}\bm{f}(x) + \bm{M}(x)\bm{f}(x)\,, \quad x\in\Rd\,.
\end{equation*}

Throughout the paper, if $\cX(\Rd)$ is a  space of real-valued functions
on $\Rd$ then we define the corresponding space
$\cX(\Rd\times\cS) \df \bigl(\cX(\Rd)\bigr)^N$, and endow it with the
product topology, if applicable.
Thus, a function $f\in\cX(\Rd\times\cS)$ is identified with the
vector-valued function
\begin{equation}\label{E-vec}
\bm{f} \,\df\, (f_1,\dotsc,f_n)\in\bigl(\cX(\Rd)\bigr)^N\,,\quad\text{where\ }
f_k(\cdot)\df f(\cdot,k)\,,\quad k\in\cS\,.
\end{equation}
With a slight abuse in notation we write $\bm{f}\in\cX(\Rd\times\cS)$.
Naturally, inequalities such as $\bm{f}\ge0$ are meant to hold componentwise.
Also, the product of two functions in $\cX(\Rd\times\cS)$ should be understood
componentwise.

We also identify an irreducibility property of the matrix $\bm{M}$
which is used in many results.

%%%%%%%%%%%%%%%%%%%%%%%%%%%%%%%%%%%%%%%%%%%%%%%%%%%%%%%%%%%%%%%%%%%%%%%%%%%%%%%%%%%%%
\begin{definition}\label{D1.1}
The matrix $\bm{M}$ is \emph{irreducible} in a bounded domain $D$ if for any non-empty
sets $\cS_1, \cS_2\subset \cS$
satisfying $\cS_1\cap\cS_2=\emptyset$ and $\cS_1\cup\cS_2=\cS$, there exists $i\in\cS_1$
and $j\in\cS_2$ satisfying
\begin{equation}\label{irred}
\babs{\{x\in D\,\colon m_{ij}(x)>0\}}\,>\,0\,.
\end{equation}
We say that $\bm{M}$ is irreducible in an unbounded domain if it is so
in some bounded subdomain.
\end{definition}

Throughout the rest of the paper, with the exception of \cref{T1.1}
we assume that $\bm{M}$ is irreducible on $\Rd$.

For a second order elliptic operator which is the extended generator
of a diffusion process, the ergodic properties of the diffusion
and the corresponding twisted diffusion play a crucial role in the
study of the eigenvalue problem.
This was thoroughly investigated in \cite{AB18,ABS19}.
In this paper, we wish to adopt an analytical approach and avoid, for the most
part, probabilistic arguments.
First, it avoids imposing unnecessary regularity hypotheses on the coefficients
of the operator to ensure the existence of an associated stochastic process.
Second by abstracting probabilistic properties into analytical ones,
the results are not restricted to an operator of the form \cref{E-bLg}
but apply to a much larger class of elliptic operators which satisfies
these properties.
Third, by `translating' the probabilistic arguments into analytical ones,
it makes the arguments more accessible to the pde community, and,
leads to a unified treatment of the problem.

\cref{D1.2} below abstracts the notions of regularity, recurrence,
and geometric ergodicity of a diffusion to analogous notions for an operator.
For nondegenerate diffusions on $\Rd$ the first two abstractions date back to the
work of Hasminskii \cite{Hasminski-60}.
Here, we develop them further.

Before stating this definition, we comment on the Dirichlet problem for the
operator $\bLg$.
Let $\cS_1$ be a nonempty subset of $\cS$, and $D_i$, $i\in\cS_1$, bounded domains
with smooth boundary.
Let
\begin{equation}\label{E-ok}
\Omega\,\df\,\cup_{i\in\cS_1} \bigl(D_i\times\{i\}\bigr)\,,\quad\text{and\ \ }
K\,\df\, \bigl(\overline{\cup_{i\in\cS_1} D_i}\times\cS\bigr)
\setminus \Omega\,.
\end{equation}
If $\bm{g}\colon K\to\RR$ is a continuous function
and $f\in\Lp^p(\Omega)$,
then the Dirichlet problem $\bLg \bm{u} = \bm{f}$ on $\Omega$, and $\bm{f}=\bm{g}$ on $K$ has
a unique solution in
$\bm{u}\in\Sobl^{2,p}(\Omega)\cap C(\Omega\cup K)$.
This is shown in \cref{L2.1} in a slightly different form.

The definition below applies to a general operator $\bLg$ on $C^2(\Rd\times\cS)$,
not necessarily of the form \cref{E-bLg}.
For example, it applies to elliptic operators containing a nonlocal component.

%%%%%%%%%%%%%%%%%%%%%%%%%%%%%%%%%%%%%%%%%%%%%%%%%%%%%%%%%%%%%%%%%%%%%%%%%%%%%%%%%%%%%
\begin{definition}\label{D1.2}
Let $D\subset\Rd$ be a bounded domain with smooth boundary
and $\cS_1$ a nonempty subset of $\cS$.
We say that $D\times\cS_1$ is \emph{recurrent} (for $\bLg$) if
the Dirichlet problem 
$$\bLg \bm{u} \,=\, 0\quad  \text{in\ } \bigl(\bar{D}\times\cS_1\bigr)^c\,,
\qquad \bm{u} \,=\, \bm{g} \quad\text{in\ }\Bar{D}\times\cS_1$$
for any $\bm{g}\in C(\Bar{D}\times\cS_1)$
has a unique bounded solution
$\bm{u}\in\Sobl^{2,p}\bigl((\bar{D}\times\cS_1)^c\bigr)
\cap C\bigl(\Rd\times\cS)$ for $p>d$.

An operator $\bLg$ on $C^2(\Rd\times\cS)$ is called
\begin{itemize}
\item[(i)]
\emph{regular} if
the equation $\bLg \bm{u} = \bm{C}\, \bm{u}$ has no bounded positive solution for any 
constant vector $\bm{C}>0$.

\item[(ii)]
\emph{recurrent} if every set of the form $B_r(x_0)\times\cS_1$, with
nonempty $\cS_1\subset\cS$, is recurrent.

\item[(iii)]
\emph{exponentially stable} if it is regular
and there exists a $\bm{\Lyap}\in\Sobl^{2,p}(\Rd\times\cS)$, $p>d$, with $\bm{\Lyap}\ge 1$
and positive constants $\kappa_0$ and $\kappa_1$ such that
\begin{equation}\label{ED1.2A}
\bLg\bm\Lyap \,\le\, \kappa_0\Ind_{K\times\cS} -\kappa_1 \bm\Lyap\quad 
\text{in\ } \Rd\times\cS\,,
\end{equation}
for some compact set $K$.
\end{itemize}
\end{definition}

In this article we shall refer to $\bm\Lyap$ as \emph{Lyapunov function}. 
\cref{D1.2}\,(i) is also known as $L^\infty$-Liouville property and is a key 
property in the study of stochastic completeness of Riemannian manifolds \cite{GA99,PRS05}.
The above definitions are motivated from the associated probabilistic model as
can be seen from the  \cref{R1.1} below.
As mentioned earlier, we do not impose any
regularity hypotheses on the coefficients to ensure existence of a stochastic process
corresponding to the extended generator $\bLg$.

We state two versions of the strong maximum principle for the operator $\bLg$ which
we use often.
These do not require irreducibility of $\bm{M}$.

%%%%%%%%%%%%%%%%%%%%%%%%%%%%%%%%%%%%%%%%%%%%%%%%%%%%%%%%%%%%%%%%%%%%%%%%%%%%%%%%%%%
\begin{enumerate}[(P2a)]
\item[\hypertarget{P2a}{(P2a)}]
Suppose that $\bm{u}\in\Sobl^{2,d}(D\times\cS)$ satisfies
$(\bLg\bm{u})_{\Hat\imath} - c u_{\Hat\imath}\le 0$ on a domain
$D$ for some constant $c\ge0$ and $\Hat\imath\in\cS$.
Then $(x,i)\mapsto u(x,i)$ cannot attain a negative minimum in $D\times\{\Hat\imath\}$.

\medskip
\item[\hypertarget{P2b}{(P2b)}]
Let $\Omega$ and $K$ be as in \cref{E-ok}.
Suppose that $\bm{u}\in\Sobl^{2,d}(\Omega)\cap C(\Omega\cup K)$ is nonnegative on $K$ and
satisfies $(\bLg\bm{u})_i - cu_i\le 0$ on $\Omega$ for some constant
$c\ge0$ and $i\in\cS_1$.
Then $u_i$ cannot have a nonpositive local minimum in $D_i$ unless it is
equal to a constant.
\end{enumerate}

%%%%%%%%%%%%%%%%%%%%%%%%%%%%%%%%%%%%%%%%%%%%%%%%%%%%%%%%%%%%%%%%%%%%%%%%%%%%%%%%%%%
\begin{remark}\label{R1.1}
Suppose that $(X,S)$ is a (local) regime switching diffusion corresponding to the generator 
$\bLg$ 
where $X=\{X_t\}_{t\ge0}$ represents the diffusion component and $S=\{S_t\}_{t\ge0}$ represents the finite state
Markov process taking values in $\cS$. In particular, $(X, S)$ solves the associated local 
martingale problem. In fact, if $\bm{a}$ has at most quadratic growth, $\bm{b}$
has at most linear growth and $\bm{M}$ is bounded, then the
corresponding martingale problem is well-posed (this follows combining \cite{XZ18},
\cite[Theorem~5.2]{Komatsu} and \cite[Theorem~5.2]{DWS75}).
Let $B_n$ denote the ball of
radius $n$ around $0$, and $\uptau_n$
the first exit time from $B_n\times\cS$, that is,
$$\uptau_n\,\df\,\inf\,\{t>0\,\colon X_t\notin B_n\}\,.$$
The diffusion $(X, S)$ is said to be regular 
\cite[Definition~2.6]{YZ10} if 
for all $(x,i)\in\Rd\times\cS$ we have $\Prob_{x,i}(\lim_{n\to\infty} 
\uptau_n<\infty)=0$. Using the It\^{o}-Krylov formula \cite[p.~122]{Krylov} it is easily seen that 
$\bm{u}_n(x,i) = \Exp_{x,i}[e^{-C \uptau_n}]$ solves the Dirichlet problem
\begin{equation}\label{E-diraux}
\bLg \bm{u}_n \,=\, C \bm{u}_n\quad \text{in\ } B_n\times\cS\,,
\quad \text{and\ \ } \bm{u}_n =1 \quad \text{on\ } \partial B_n\times\cS\,.
\end{equation}
Suppose that $\bm{v}$ is a bounded positive solution of $\bLg\bm{v}=C\bm{v}$.
Without loss of generality, assume that $\bm{v}\le\bm{1}$.
By the strong maximum principle, $\bm{u}_n\ge\bm{v}$ in $B_n$.
Thus taking limits it follows that $\Exp_{x,i}[e^{-C \uptau_\infty}]>0$,
where $\uptau_\infty=\lim_{n\to\infty}\uptau_n$,
from which it is straightforward to deduce that $\uptau_\infty<\infty$
with positive probability. Thus the diffusion cannot be regular.
Conversely, if the diffusion is not regular, then
using standard elliptic pde estimates it can be easily shown that
$\bm{u} = \lim_{n\to\infty}\,\bm{u}_n$ is a
solution of 
$$\bLg \bm{u} \,=\, C \bm{u}\quad \text{in\ } \Rd\times\cS\,.$$
Since $\Prob(\uptau_\infty<\infty)>0$ by hypothesis,
it follows that $\bm{u}(x,i) = \Exp_{x,i}[e^{-C \uptau_\infty}]>0$,
and hence this solution
is bounded and positive.
Thus the regularity of $(X,S)$ is equivalent to the regularity
of $\bLg$ in \cref{D1.2}\,(i).
The above argument is due to Hasminskii \cite[Lemma~4.1]{Hasminski-60}.

Next, we show that nonexistence of positive bounded  solution to
$\bLg\bm{u}=\bm{u}$ characterizes regularity.

%%%%%%%%%%%%%%%%%%%%%%%%%%%%%%%%%%%%%%%%%%%%%%%%%%%%%%%%%%%%%%%%%%%%%%%%%%%%%%%%%%%
\begin{lemma}
Suppose that, for all large enough $n\in\NN$, the operator $\bLg$ satisfies the strong maximum principle
\hyperlink{P2a}{\rm(P2a)}
on $B_n\times\cS$ and the Dirichlet problem in \cref{E-diraux} has a unique solution
for every $C>0$.
Then $\bLg$ is regular if and only if
the equation $\bLg\bm{u}=\bm{u}$ has no  positive bounded solution.
\end{lemma}

\begin{proof}
Necessity follows from \cref{D1.2}\,(i).
To show sufficiency. we argue as follows.
First, suppose that $\bm{v}<\bm{1}$ is a bounded positive subsolution
of $\bLg\bm{v}\ge\kappa\bm{v}$ for some $\kappa>0$.
Let $C\in(0,\kappa]$. The solution $u_n$ of \cref{E-diraux}
satisfies
$$\bLg(\bm{u}_n-\bm{v}) - C (\bm{u}_n-\bm{v}) \,\le\,(C-\kappa)\bm{v}\,\le\,0
\quad\text{in\ } B_n\,,$$
and thus we must have
$\bm{u}_n>\bm{v}$ on $B_n$ by the strong maximum principle.
Taking limits as $n\to\infty$, we deduce that $\bLg\bm{u}=C\bm{u}$ must have
a positive bounded solution for any $C\in(0,\kappa]$.
To finish the proof, it suffices to show that $\bLg\bm{u}\ge2\kappa\bm{u}$ has
a positive bounded solution.
We argue by contradiction. Suppose that this equation has no positive bounded solution.
Consider the Dirichlet problems
$\bLg \bm{u}_n = 2\kappa \bm{u}_n$ in $B_n\times\cS$ with $\bm{u}_n=\norm{\bm{v}}_\infty$ on
$\partial{B}_n\times\cS$.
According to the hypothesis $\bm{u}_n$ must converge pointwise to $0$ as $n\to\infty$.
Let $u_{i,n} \df (\bm{u}_n)_i$ for $i\in\cS$.
Then, for all large enough $n$, we must have that
\begin{equation}\label{PL1.1A}
\inf_{B_n}\; \min_{i\in\cS}\;(u_{i,n}-v_{i})\,\le\, -\tfrac{2}{3} \norm{\bm{v}}_\infty\,.
\end{equation}
Hence if $(x_n,\Hat\imath)\in B_n\times\cS$ is a point where $\min_{i\in\cS}\,(u_{i,n}-v_{i})$
attains its infimum in $B_n$,
it follows from \cref{PL1.1A} that
\begin{equation}\label{PL1.1B}
2u_{n,\Hat\imath}(x_n)-v_{\Hat\imath}(x_n) \,\le\,
v_{\Hat\imath}(x_n) - \frac{4}{3}\norm{\bm{v}}_\infty\,<\,0\,.
\end{equation}
Since $\bigl(\bLg (\bm{u}_n-\bm{v})\bigr)_{\Hat\imath}
\le \kappa (2u_{n,\Hat\imath}-v_{\Hat\imath})$
in $B_n$, \cref{PL1.1B} shows that
$\bigl(\bLg (\bm{u}_n-\bm{v})\bigr)_{\Hat\imath}<0$ in some open neighborhood of
$x_n$ which contradicts \hyperlink{P2a}{\rm(P2a)}.
\end{proof}

We can also relate \cref{D1.2}\,(ii) with the recurrence property of a process
$(X, S)$ associated with $\bLg$.
Recall that a bounded domain $D\times\cS_1$ is said to be recurrent
for the process $(X,S)$ if $\Prob_{x,i}(\uuptau(D\times\cS_1)<\infty)=1$ for all
$(x,i)\in(D\times\cS_1)^c$ where $\uuptau(A)$ denotes the first hitting time 
of the set $A$, that is,
$$\uuptau(A)\,\df\,\inf\{t>0\,\colon (X_t, S_t)\in A\}\,.$$
Also, a process $(X,S)$ is said to be recurrent if
any such domain $D\times\cS_1$ is recurrent.
With the help of It\^{o}-Krylov formula it can be easily checked that 
$\bm{u}(x,i)=\Prob_{x,i} (\uuptau(D\times\cS_1)<\infty)$ is a solution to
$$\bLg \bm{u}=0\quad  \text{in\ } (\bar{D}\times\cS_1)^c,\quad
\bm{u}=1\quad \text{in\ } \Bar{D}\times\cS_1.$$
One can easily relate the recurrence of $(X, S)$ with \cref{D1.2}\,(ii).
This is also discussed in \cite{Hasminski-60}.
We want to note here that the recurrence properties depend on the irreducibility of $\bm{M}$
in \cref{D1.1}.
If $\bm{M}$ is irreducible on $\Rd$, then for the process  $(X,S)$ to be recurrent
it suffices that any bounded domain $D\times\cS_1$ is recurrent.
Also, we have a dichotomy: the process is either recurrent of \emph{transient}.
However, if $\bm{M}$ is reducible on $\Rd$, this dichotomy does not hold.
For a counterexample see \cref{Ex1.1} below.
\end{remark}

The following example shows that exponential stability as in
\cref{D1.2}\,(iii) does not immediately imply recurrence, unless
$\bm{M}$ is irreducible on $\Rd$.

\begin{example}\label{Ex1.1}
Let $\cS=\{1,2\}$, and dynamics given by
\begin{align*}
\D{X}_1(t) &\,=\, \sign(X_{1}(t)) + \sqrt2\, \D{W}_1 (t)\,,\\
\D{X}_2(t) &\,=\, -X_{2}(t) + \sqrt2\, \D{W}_2 (t)\,.
\end{align*}
Here, $X_1$ and $X_2$ are $1$-dimensional It\^o processes.
Suppose that $\bm{M}=0$ on $B_2\times\cS$, while on
$B_2^c\times\cS$ we have
$m_{11}=-\delta$, $m_{12}=\delta$, $m_{21}=-m_{22}=0$, for some $\delta>0$.
It can be seen, by using the Lyapunov function $\Lyap_2(x) = \frac{x^2}{1+x^2}$ and
$\Lyap_1(x) = 3 - \Lyap_2{x}$, the switched diffusion is exponentially stable
according to \cref{D1.2}\,(iii).
However, it is clear that the set $B\times\{1\}$ cannot be recurrent.
The lack of irreducibility of $\bm{M}$ is responsible for this.
On the other hand, the set $B\times\cS$ is recurrent.
If we modify $\bm{M}$, and let $m_{21}=-m_{22}=\epsilon$ in
$B_2^c\times\cS$ for some $\epsilon>0$,
then the Lyapunov equation \cref{ED1.2A} still holds,
and $B\times\{1\}$ is now recurrent.
\end{example}

\Cref{T1.1} below, concerns the relations among (i)--(iii) in \cref{D1.1}.
These are well-known if $\bLg$ is the extended generator of a stochastic
process 
$(X,S)$ and $\bm{M}$ is irreducible on $\Rd$ (cf. \cite{book,Pinsky,YZ10}).
Our objective though is to provide analytical proofs in a very general setting
without using the probabilistic structure.
The results apply to any elliptic operator $\bLg$ satisfying
the strong maximum principle in \hyperlink{P2a}{(P2a)} or
\hyperlink{P2b}{(P2b)} and for which the Dirichlet problem on
a bounded domain $D\times\cS_1$ has a solution.

\begin{theorem}\label{T1.1}
The following hold.
\begin{itemize}
\item[(a)] A recurrent operator $\bLg$ is regular.
\item[(b)] Provided that $\bm{M}$ is irreducible on $\Rd$,
an exponentially stable operator $\bLg$ is recurrent.
\item[(c)] Irrespective of the irreducibility properties of $\bm{M}$,
if $\bLg$ is exponentially stable, then any bounded domain of the
form $D\times\cS$ is recurrent.
\end{itemize}
\end{theorem}

The proof of \cref{T1.1} is in \cref{S2}.

As already mentioned, this paper is devoted to the study of the
generalized principal eigenvalue $\lamstr $ in $\Rd$ of $\bm\cA$.
Let
\begin{equation*}
\Uppsi^+(\lambda) \,\df\, \bigl\{\bm{f}\in \Sobl^{2,d}(\Rd\times\cS)
\,\colon\, \bm{f}> 0 \,,
\ \bcA\,\bm{f} + \lambda\,\bm{f} \le\, 0 \;\text{\ in\ } \Rd \bigr\}\,,\quad 
\lambda\in\RR\,.
\end{equation*}
The principal eigenvalue $\lamstr$ is defined as
\begin{equation}\label{E1.9}
\lamstr  \,\df\, \sup\,\bigl\{\lambda\in\RR\,\colon\,
\Uppsi^+(\lambda)\ne \varnothing\bigr\}\,.
\end{equation}
We refer to the parameter $\bm{c}$ as the \emph{potential} and, when needed,
we indicate the dependence of $\lamstr$ on $\bm{c}$ explicitly in the notation
by writing $\lamstr(\bm c)$. Some early works on generalized principal eigenvalue
for scalar elliptic equation appeared in Protter--Weinberger \cite{PW66}, Nussbaum 
\cite{N84} and Nussbaum--Pinchover \cite{NP92}.
Generalized eigenvalues and its relation to maximum principles in bounded domains are 
established in the seminal work
of Berestycki--Nirenberg--Varadhan \cite{Berestycki-94}. Later, this was extended to
more general 
operators. Recently, Berestycki--Rossi \cite{Berestycki-15} studied the principal eigenvalue 
problem for scalar elliptic operators
in unbounded domains and established several interesting properties (see also \cite{NP92}).

We say that a constant $\lambda\in\RR$ and a positive $\bPsi\in\Sobl^{2,p}(\Rd\times\cS)$, $p>d$, 
solve the eigenvalue equation for $\bcA$, if
$\bcA\,\bPsi =-\lambda\,\bPsi$.
In such a case we call $\bPsi$ the eigenfunction
and $(\bPsi,\lambda)$ the \emph{eigenpair}.
An eigenfunction is always meant to be a positive function.
The theorem which follows is also a generalization of \cite[Theorem~1.4]{Berestycki-15}.

%%%%%%%%%%%%%%%%%%%%%%%%%%%%%%%%%%%%%%%%%%%%%%%%%%%%%%%%%%%%%%%%%%%%%%%%%%%%%%%%
\begin{theorem}\label{T1.2}
For any $\lambda\le\lamstr$ there exists $\bPsi\in\Sobl^{2,p}(\Rd\times\cS)$,
$\bPsi> 0$, satisfying
$\bm\cA\bPsi = -\lambda\bPsi$ in $\Rd\times\cS$.
\end{theorem}
In view of \cref{T1.2}, the following question seems natural.

\medskip\noindent
\textbf{Question.}\ Given an eigenpair $(\bPsi, \lambda)$, when can we identify it as a 
principal eigenpair?

\medskip
A main goal of this article is to answer this question by 
exploiting the `stability' properties of the twisted operators.
Recall the operator $\bLg$ in \cref{E-bLg}.
Corresponding to an eigenpair $(\bPsi, \lambda)$, we let $\bm\psi\df\log\bPsi$ 
componentwise,
and define the twisted operator $\widetilde\bLg^{\bm\psi}$ as follows.
We first define the operators $\bigl\{L_k^{\bm\psi}\bigr\}_{k\in\cS}$, by
\begin{equation*}%\label{E-tLk}
L_k^{\bm\psi} g (x) \,\df\, \trace\bigl(a_k(x)\grad^2 g(x)\bigr)
+ \bigl(b_k(x)+ 2a_k(x)\nabla\psi_k(x)\bigr)\cdot \grad g(x)\,,
\quad x\in\Rd\,,\quad g\in C^2(\Rd)\,,
\end{equation*}
and the matrix
$\widetilde{\bm{M}} = \bigl[\widetilde{m}_{ij}\bigr]_{i,j\in\cS}$
by
\begin{equation*}
\widetilde{m}_{ij} \,\df\, m_{ij}\frac{\Psi_j}{\Psi_i}\text{\ \ for\ } i\ne j\,,
\quad\text{and\ }
\widetilde{m}_{ii} \,\df\, -\sum_{j\in\cS,\,j\ne i} \widetilde{m}_{ij}\,.
\end{equation*}
With these definitions, the twisted operator is given by
\begin{equation}\label{E-tbLg}
\widetilde\bLg^{\bm\psi} \bm{f} (x) \,\df\, \bm{L}^{\bm\psi}\bm{f}(x)
+ \widetilde{\bm{M}}(x)\bm{f}(x)\,, \quad x\in\Rd\,.
\end{equation}

Let us first consider the case of bounded domains. Let $D$ be a smooth
bounded domain in $\Rd$, and $(\bPsi_D, \lambda_D)$ be the Dirichlet principal eigenpair of 
$\bcA$ in $D$, that is,
\begin{equation}\label{E1.12}
\begin{split}
\bcA \bPsi_{D} &\,=\, \lambda_D\, \bPsi_{D} \quad \mbox{in\ }
D\times\cS\,,\\
\bPsi_{D} &\,=\, 0\quad \text{on\ } \partial{D}\times\cS\,,\\
\bPsi_{D} &\, > \, 0 \quad \text{in\ } D\times\cS\,,
\end{split}
\end{equation}
and $\bPsi_D\in \cC_{0}(\overline{D}\times\cS)
\cap\Sobl^{2,p}(D\times\cS)$ for $p>d$ (see \cref{TA.1}).
We note that $\lambda_D$ is the only eigenvalue with a positive eigenfunction.
We let $\widetilde{\bLg}_D$ denote the twisted operator corresponding to 
the eigenpair $(\bPsi_D, \lambda_D)$. Then we have the following result.

%%%%%%%%%%%%%%%%%%%%%%%%%%%%%%%%%%%%%%%%%%%%%%%%%%%%%%%%%%%%%%%%%%%%%%%%%%%%%%%%%%%%%
\begin{theorem}\label{T1.3}
There exists a inf-compact function 
$\bm{\Lyap}_D: D\times \cS\to [1, \infty)$, $\bm\Lyap_D\in \Sobl^{2,p}(D\times\cS)$, 
satisfying
\begin{equation}\label{ET1.3A}
\widetilde{\bLg}_D\bm\Lyap_D\,\le -\delta_1\bm{\Lyap}_D + \delta_2\quad \text{in\ } 
D\times\cS
\end{equation}
for some constants
$\delta_1, \delta_2>0$.
Furthermore, $\widetilde\bLg_D$ is regular when restricted to $D$ in the sense of
\cref{D1.2}\,\ttup{i}\,.
\end{theorem}
The function $\bm\Lyap_D$ in \eqref{ET1.3A} is commonly known as Lyapunov function. It 
turns out that the existence of a Lyapunov function in a bounded domain follows from the 
monotonicity property (cf. \cref{TA.2,TA.3}) of the
principal eigenvalue; this does not always hold in $\Rd$.

%%%%%%%%%%%%%%%%%%%%%%%%%%%%%%%%%%%%%%%%%%%%%%%%%%%%%%%%%%%%%%%%%%%%%%%%%%%%%%%%%%%%%
\begin{remark}
The process associated with $\widetilde{\bLg}_D$ in \cref{T1.3}
is confined in the domain $D$, and is known
in the literature as the \emph{$Q$-process}.
There is an extensive literature on the $Q$-process covering
various classes of Markov processes.
We cite here \cite{Pinsky-85,Champagnat-16,Champagnat-17}.
\end{remark}

%%%%%%%%%%%%%%%%%%%%%%%%%%%%%%%%%%%%%%%%%%%%%%%%%%%%%%%%%%%%%%%%%%%%%%%%%%%%%%%%%%%%%
\begin{remark}
Lyapunov functions play a central role in the study of exponential ergodicity of
regime switching diffusions.
In fact, finding sufficient condition for the existence of Lyapunov function is an 
important issue.
See, for instance, \cite{CH15,HMS11,YZ10,XZ06,TM16} and references therein.
In particular, by \cite[Theorem~5.3]{XZ06}, if $\bm\Lyap$ is inf-compact,
$\bm{M}$ is bounded and $\bm{a}$ is 
uniformly elliptic,
we get $\bm\Lyap$-geometric ergodicity for the regime switching diffusion.
In this paper,
we only concentrate on the existence of Lyapunov functions and do not address the delicate 
issue of exponential ergodicity.
\end{remark}

The existence of a Lyapunov function for $\widetilde\bLg$ in $\Rd$ is related to a certain 
monotonicity property of the principal eigenvalue in $\Rd$,
which we state next. By $\sB_0^+(\Rd\times\cS)$ we 
denote the class of all nontrivial,
nonnegative bounded measurable functions
$\bm h\colon\Rd\times\cS\to [0,\infty)$ that vanish at infinity.

%%%%%%%%%%%%%%%%%%%%%%%%%%%%%%%%%%%%%%%%%%%%%%%%%%%%%%%%%%%%%%%%%%%%%%%%%%%%%%%%%%%%%
\begin{definition}\label{D1.3}
We say $\lamstr$ is \textit{monotone at $\bm c$ on the right}, if
$\lamstr(\bm c) > \lamstr(\bm{c} + \bm{h})$
for all $\bm h \in\sB^+_0(\Rd\times\cS)$, and
we say $\lamstr$ is \emph{strictly monotone at $\bm c$}, if
$\lamstr(\bm{c} - \bm{h}) > \lamstr(\bm c)$
for some $\bm h \in \sB^+_0(\Rd\times\cS)$.
\end{definition}

%%%%%%%%%%%%%%%%%%%%%%%%%%%%%%%%%%%%%%%%%%%%%%%%%%%%%%%%%%%%%%%%%%%%%%%%%%%%%%%%%%%%%
\begin{remark}
It is 
shown later in
\cref{T1.6} that strict monotonicity implies $\lamstr(\bm{c} - \bm{h}) > \lamstr(\bm c)$
for all $\bm h \in \sB^+_0(\Rd\times\cS)$. Furthermore, since 
$\bm{c}\mapsto\lamstr(\bm{c})$ 
is concave
by \cref{L2.3}, strict monotonicity implies monotonicity on the right.
\end{remark}

Throughout the rest of the paper, we always assume that the principal eigenvalue
is finite:

%%%%%%%%%%%%%%%%%%%%%%%%%%%%%%%%%%%%%%%%%%%%%%%%%%%%%%%%%%%%%%%%%%%%%%%%%%%%%%%%%%%%%
\begin{hypothesis}\label{H1.1}
$\lamstr=\lamstr(\bm{c})<\infty$.
\end{hypothesis}

The next result shows that twisted operators corresponding to the lower eigenvalues are 
not recurrent. This should be compared with 
\cite[Proposition~3.3]{Kaise-06} and \cite[Theorem~2.1]{ABS19}.

%%%%%%%%%%%%%%%%%%%%%%%%%%%%%%%%%%%%%%%%%%%%%%%%%%%%%%%%%%%%%%%%%%%%%%%%%%%%%%%%%%%%%
\begin{theorem}\label{T1.4}
For any $\lambda < \lamstr$, the twisted operator $\widetilde{\bLg}$
corresponding to $(\bPsi,\lambda)$ is not recurrent.
\end{theorem}

This of course, brings us to the question what happens to the twisted operator 
corresponding to
the principal eigenpair. 
As the following example suggests, the twisted operator may be
non-recurrent for the principal eigenpair even when $\bm{M}$ is irreducible.

%%%%%%%%%%%%%%%%%%%%%%%%%%%%%%%%%%%%%%%%%%%%%%%%%%%%%%%%%%%%%%%%%%%%%%%%%%%%%%%%%%%%%
\begin{example}
Let $N=2$ and $a=\Id$, $b=0$, $c=0$, and $m_{12}=m_{21}=1$.
Then the constant functions are the principal eigenfunctions and $\lamstr=0$.
Note that $\bigl(\sin(n^{-1}\pi x), \sin(n^{-1}\pi x)\bigr)$ are the eigenfunctions
in $B_n\times\cS$ with eigenvalue $\frac{\pi^2}{n^2}$. Therefore, 
$\lamstr=\lim_{n\to\infty}\frac{\pi^2}{n^2}=0$.
Since constants are principal eigenfunctions, the corresponding twisted operator is
the same as $\bLg$.
Therefore, the twisted operator is not recurrent for $d\ge 3$. Indeed,
setting $u_i(x)=1-|x|^{2-d}$ and $B=B_1$ we see that
$$\bLg \bm{u} \,=\, 0 \quad \text{in\ } (B\times\cS)^c, \quad \text{and}
\quad \bm{u} \,=\, 0\quad \text{on\ } \partial B\times \cS.$$
\end{example}

It turns out that the recurrence of the twisted operator is equivalent to the monotonicity 
of $\lamstr$ on the right.

%%%%%%%%%%%%%%%%%%%%%%%%%%%%%%%%%%%%%%%%%%%%%%%%%%%%%%%%%%%%%%%%%%%%%%%%%%%%%%%%%%%%%
\begin{theorem}\label{T1.5}
The following are equivalent.
\begin{itemize}
\item[(a)] The twisted operator corresponding to $(\bPsi^*, \lamstr)$ is recurrent.
\item[(b)] $\lamstr$ is monotone on the right at $\bm c$.
\end{itemize}
In addition, under either \ttup{a} or \ttup{b}, $\lamstr$ is a simple eigenvalue.
\end{theorem}

The next result characterizes the strict monotonicity property.
\begin{theorem}\label{T1.6}
The following are equivalent.
\begin{itemize}
\item[(a)] The twisted operator corresponding to $(\bPsi^*, \lamstr)$ is exponentially 
stable.
\item[(b)] $\lamstr$ is strictly monotone at $\bm c$.
\item[(c)] For any $\bm h\in\sB^+_0(\Rd\times\cS)$, we have $\lamstr(\bm c - \bm h) > 
\lamstr(\bm c)$.
\end{itemize}
\end{theorem}

It is interesting to observe from \cref{T1.5,T1.6} that monotonicity of the
principal eigenvalue implies simplicity of the eigenvalue. Another criterion that is 
often used to ensure simplicity of principal eigenvalue is
\emph{Agmon's minimal growth at infinity} introduced by Agmon in \cite{Agmon83} (see also 
\cite[Definition~8.2]{Berestycki-15} and \cite{Pinchover-88,Pinchover-89}).

%%%%%%%%%%%%%%%%%%%%%%%%%%%%%%%%%%%%%%%%%%%%%%%%%%%%%%%%%%%%%%%%%%%%%%%%%%%%%%%%%%%%%
\begin{definition}[Minimal growth at infinity]
An eigenpair $(\bPsi,\lambda)$ is said to have minimal growth at infinity,
if for any compact set $K\times\cS_1\subset\Rd\times\cS$ and 
for any $\bPhi\in\Sobl^{2,p}((K\times\cS_1)^c)$, $p>d$,
continuous and positive in $\Rd\times\cS$,
satisfying
\begin{equation*}\bLg\bPhi + (\bm{c}+\lambda)\bPhi \,\le\,  0 \quad \text{in\ } 
(K\times\cS_1)^c\,,
\end{equation*}
we have $\bPhi\ge \kappa \bPsi$ in $\Rd\times\cS$, for some $\kappa>0$.
\end{definition}

Then the following result is immediate.

%%%%%%%%%%%%%%%%%%%%%%%%%%%%%%%%%%%%%%%%%%%%%%%%%%%%%%%%%%%%%%%%%%%%%%%%%%%%%%%%%%%%%
\begin{theorem}\label{T1.7}
Suppose that $(\bPsi,\lambda)$ has the minimal growth at infinity.
Then $(\bPsi,\lambda)$ is a principal eigenpair and $\lambda$ is simple.
\end{theorem}

\begin{proof}
Let $(\bPsi^*,\lamstr)$ be a principal eigenpair.
Then it follows from \cref{T1.2} that $\lamstr\ge \lambda$, and 
\begin{equation*}
\bLg\bPsi^* + (\bm{c}+\lambda)\bPsi^* \,\le\, 
\bLg\bPsi^* + (\bm{c}+\lamstr)\bPsi^*\,=\,  0 \quad \text{in\ } \Rd\,.
\end{equation*}
Thus, the minimal growth at infinity of
$(\bPsi,\lambda)$ implies that $\bPsi^*> \kappa \bPsi$ for some $\kappa>0$. Let 
\begin{equation*}
\kappa_1 \,\df\, \min_{k\in\cS}\, \inf_{\Rd}\,\frac{\Psi^*_k}{\Psi_k}\,.
\end{equation*}
We claim that $\bPsi^*- \kappa_1\bPsi\ge 0$, and that at least one of the components
must vanish in $\Rd$. If not, then we get $\bPhi=\bPsi^*- \kappa_1\bPsi> 0$ and
\begin{equation*}
\bLg\bPhi + (\bm{c}+\lambda)\bPhi \,\le\,  0\,, \quad \text{in\ } \Rd\,,
\end{equation*}
which implies that $\bPhi> \kappa_2\bPsi$ for some $\kappa_2>0$.
But this contradicts the definition of $\kappa_1$. Thus, one
of the components of $\bPhi$ must vanish in $\Rd$.
The strong maximum principle then implies that $\bPhi=0$.
Hence $\lambda=\lamstr$ and
$\bPsi^*=\kappa_1\bPsi$ in $\Rd\times\cS$. This completes the proof.
\end{proof}

Continuing, we show that minimal growth at infinity is equivalent to
monotonicity of the principal eigenvalue on the the right. For
the scalar equation an analogous result is established in
 \cite{ABG19} using probabilistic methods.
In contrast, the proof of \cref{T1.8} is
 analytical, and thus more general in nature.

%%%%%%%%%%%%%%%%%%%%%%%%%%%%%%%%%%%%%%%%%%%%%%%%%%%%%%%%%%%%%%%%%%%%%%%%%%%%%%%%%%%%%
\begin{theorem}\label{T1.8}
The following are equivalent.
\begin{itemize}
\item[(a)]
$\lamstr$ is monotone on the right at $c$.
\item[(b)]
$(\Psi^*, \lamstr)$ has minimal growth at infinity.
\end{itemize}
\end{theorem}

Next we relate the monotonicity property of $\lamstr$ on the right with the stochastic
representation of the principal eigenfunction $\bPsi^*$. To do so we need to impose
mild restrictions on the coefficients of $\bLg$ to ensure the existence of a
strong solution.
\subsubsection{Description of the probabilistic model}\label{S-prob}
We introduce the regime switching diffusion process.
This is a process $(X_{t}, S_{t})$ in $\Rd\times \cS$
governed by the following stochastic differential equations:
\begin{equation}\label{E1.14}
\begin{aligned}
\D X_t &\,=\, b(X_t,S_t) \D t + \upsigma(X_t,S_t)\, \D W_t\,,\\
\D S_t &\,=\, \int_{\RR} h(X_t,S_{t^-},z)\wp(\D t, \D z)\,,
\end{aligned}
\end{equation}
for $t\ge 0$. Here
\begin{itemize}
\item[(i)]
$S_{0}$ is a prescribed $\cS=\{1,2,\dotsc,N\}$ valued random variable;
\item[(ii)]
$X_{0}$ is a prescribed $\Rd$ valued random variable;
\item[(iii)]
$W$ is a $d$-dimensional standard Wiener process;
\item[(iv)]
$\wp(\D t,\D z)$ is a Poisson random measure on $\RR_{+}\times\RR$ with intensity
$\D t\times \mu(\D z)$, where $\mu$ is the Lebesgue measure on $\RR$;
\item[(v)]
$\wp(\cdot,\cdot)$, $W(\cdot)$, $X_{0}$, and $S_{0}$ are independent;
\item[(vi)]
The function $h\colon\Rd\times\cS \times \RR \to \RR$ is defined by
\begin{equation*}
h(x,i,z)\,\df\,\begin{cases}
j - i & \text{if}\,\, z\in \Vt_{ij}(x),\\[1mm]
0 & \text{otherwise},
\end{cases}
\end{equation*}
where for $i,j\in\cS$ and fixed $x$, $\Vt_{ij}(x)$ are left closed right open
disjoint intervals of $\RR$ having length $m_{ij}(x)$.
\end{itemize}

Note that $\bm{M}(x)$ can be interpreted as the rate matrix of the
Markov chain $S_{t}$ given that $X_t=x$. In other words, 
\begin{equation*}
\Prob(S_{t+h}=j\, |\, X_t, S_t) \,=\, \begin{cases}
m_{S_t j}(X_t)h + \sorder(h) & \text{if\ } S_t\neq j\,,
\\[2mm]
1+ m_{S_t j}(X_t)h + \sorder(h) & \text{if\ } S_t = j\,,
\end{cases}
\end{equation*}
and $X$ behaves like an ordinary diffusion process governed by \eqref{E1.14} between two consecutive 
jumps of $S$. In addition to \eqref{irred},
we impose the following assumptions to guarantee existence of solution of \cref{E1.14}.

%%%%%%%%%%%%%%%%%%%%%%%%%%%%%%%%%%%%%%%%%%%%%%%%%%%%%%%%%%%%%%%%%%%%%%%%%%%%%%%%%%%%%
\begin{itemize}
\item[\hypertarget{A1}{{(A1)}}]
\emph{Local Lipschitz continuity:\/}
The function
$\upsigma=\bigl[\upsigma^{ij}\bigr]\colon\RR^{d}\times\cS\to\RR^{d\times d}$
is continuous and locally Lipschitz in $x$ with a Lipschitz constant $C_{R}>0$
depending on $R>0$.
In other words, with $\norm{\upsigma}\df\sqrt{\trace(\upsigma\upsigma\transp)}$,
we have
\begin{equation*}
\norm{\upsigma(x,k) - \upsigma(y,k)}^2
\,\le\, C_{R}\,\abs{x-y}^2 \qquad\forall\,x,y\in B_R\,,\ \forall\,k\in\cS\,.
\end{equation*}
The function $b\colon\Rd\times\cS\to\Rd$
is assumed to be Borel measurable and locally bounded.

\medskip
\item[\hypertarget{A2}{{(A2)}}]
\emph{Affine growth condition:\/}
$b(x,k)$ and $\upsigma(x,k)$ satisfy a global growth condition of the form
\begin{equation*}
\langle b(x,k),x\rangle^{+} + \norm{\upsigma(x,k)}^{2} \,\le\,C_0
\bigl(1 + \abs{x}^{2}\bigr) \qquad \forall\, (x,k)\in\Rd\times\cS\,,
\end{equation*}
for some constant $C_0>0$.

\medskip
\item[\hypertarget{A3}{{(A3)}}]
\emph{Nondegeneracy:\/}
For each $R>0$, it holds that
\begin{equation*}
\sum_{i,j=1}^{d} a_k^{ij}(x)\zeta_{i}\zeta_{j}
\,\ge\,C^{-1}_{R} \abs{\zeta}^{2} \qquad\forall\, (x,k)\in B_R\times\cS\,,
\end{equation*}
and for all $\zeta=(\zeta_{1},\dotsc,\zeta_{d})\transp\in\RR^{d}$,
where, $a\df \frac{1}{2}\upsigma \upsigma\transp$.
\end{itemize}

It is well known that under hypotheses
\hyperlink{A1}{{(A1)}}--\hyperlink{A3}{{(A3)}},
\cref{E1.1} has a unique strong solution with $X\in C(\RR_{+};\Rd)$,
and $S_{t}\in \mathcal{D}(\RR_{+};\cS)$, where $\mathcal{D}(\RR_{+};\cS)$
denotes the space of all right continuous functions from $\RR_{+}$ to $\cS$
having left limit \cite{AGM93} (cf.\ \cite[Remark~5.1.2]{book}).
Moreover, the solution $(X_t,S_t)$ is a  Feller process
(see \cite[Theorem~2.1]{AGM93} and \cite[Remark~5.1.6]{book}) and therefore,
a strong Markov process. Also, the ergodic behavior of $Y_t \df (X_t,S_t)$ depends 
heavily on the coupling coefficients
$\{m_{ij}\}$ (cf.\ \cite{book}, \cite[Chapter~2]{YZ10}).

For a ball $\sB$, centered at $0$, we denote by
$\uuptau$ the first hitting time to $\sB\times\cS$, that is,
$$\uuptau \,\df\, \inf\{t>0\,\colon X_t\in \sB\}\,.$$
\Cref{T1.9}, which follows, asserts the equivalence between monotonicity on the right and
a stochastic representation of $\bPsi^*$.

%%%%%%%%%%%%%%%%%%%%%%%%%%%%%%%%%%%%%%%%%%%%%%%%%%%%%%%%%%%%%%%%%%%%%%%%%%%%%%%%%%%%%
\begin{theorem}\label{T1.9}
The following are equivalent.
\begin{itemize}
\item[(a)]
$\lamstr$ is monotone on the right at $c$.
\item[(b)]
For some ball $\sB$ we have
\begin{equation}\label{ET2.7A}
\Psi^*_k(x) \,=\,  \Exp_{x, k}
\left[\E^{\int_0^{\uuptau}(\bm{c}(X_t, S_t)-\lamstr)\, \D{t}}
\bPsi^*(X_{\uuptau}, S_{\uuptau})\Ind_{\{\uuptau<\infty\}}\right],
\quad (x,k)\in\sB^c\times\cS\,.
\end{equation}
\end{itemize}
\end{theorem}

Before we conclude this section, let us compare the contribution
of this paper with the existing work. 
The notion of monotonicity was introduced in the \cite{ABS19}, and results analogous to 
\cref{T1.4,T1.5,T1.6} were proved in for a scalar operator using probabilistic 
methods. In particular, $a$ was assumed to be locally Lipschitz, $b$ was assumed to satisfy 
\hyperlink{A2}{{(A2)}} and $c$ was assumed to be bounded from below.
In this article we do not impose such restrictions.
So the arguments in \cite{ABS19} do not work for us in this article.
As can be seen, \cref{T1.9} is the only result that relies
on the probabilistic model,
but the proof does not use the results in \cite{ABS19}.
This is because of 
the nonavailability of a suitable Girsanov transformation for a general
regime switching diffusion.
Instead, we study the parabolic system (see \cref{L2.4}) to find a
substitute for Girsanov's transformation for this model.
In this manner, we obtain an explicit form for the twisted operator
in \cref{E-tbLg} for elliptic systems.

%%%%%%%%%%%%%%%%%%%%%%%%%%%%%%%%%%%%%%%%%%%%%%%%%%%%%%%%%%%%%%%%%%%%%%%%%%%%%%%%%%%%%
\subsection{Notation}\label{S-not}
$\bm F > \kappa$ would mean $F_{k} > \kappa\,$ for all $k\in\cS$ 
and $\bm F \ge 0$ means $F_{k} \ge 0$ for all $k\in\cS$. $\bm F \gneq 0$ means $F_{k} \ge 
0$ for all $k\in\cS$
and $\sum_{k\in\cS} F_k > 0$ on a set of positive Lebesgue measure in $\Rd$. By
$B_r(x)$ we denote the ball of radius $r$ around $x$ and for $x=0$ we simply denote it by 
$B_r$.

If $\cX(Q)$ is a topological space of real-valued functions on a domain $Q\subset\Rd$,
we denote by $\cX(Q\times\cS)$ the space $\bigl(\cX(Q)\bigr)^N$ endowed with
the product topology inherited from $\cX(Q)$.
As already explained in \cref{E-vec},
if $f$ is a real valued function on $Q\times\cS$, we let
$f_k(\cdot)\df f(\cdot,k)$, and identify $f$ with $\bm{f} \df (f_1,\dotsc,f_N)$,
which is viewed as a vector-valued function on $Q$.
If $\cX(Q)$ is endowed with a norm $\norm{\cdot}_{\cX(Q)}$, we let
$\norm{f}_{\cX(Q\times\cS)}\df \sum_{k\in\cS} \norm{f_k}_{\cX(Q)}$
for $f\in\cX(Q\times\cS)$.

%We adopt the notation
%$\partial_{i}\df\tfrac{\partial~}{\partial{x}_{i}}$ and
%$\partial_{ij}\df\tfrac{\partial^{2}~}{\partial{x}_{i}\partial{x}_{j}}$
%for $i,j\in\{1,\dotsc,d\}$, and we
%often use the standard summation rule that
%repeated subscripts and superscripts are summed from $1$ through $d$.
%We also often use Krylov's extension of It\^{o}'s formula for functions
%in $\Sobl^{2,d}(\Rd)$ which we refer to as the
%It\^o--Krylov formula \cite[p.~122]{Krylov}.

%%%%%%%%%%%%%%%%%%%%%%%%%%%%%%%%%%%%%%%%%%%%%%%%%%%%%%%%%%%%%%%%%%%%%%%%%%%%%%%%%%%%
\section{Proofs of main results}\label{S2}
In this section we present the proofs of the main results. 
The proof of \cref{T1.1} requires the following
Liouville property.

%%%%%%%%%%%%%%%%%%%%%%%%%%%%%%%%%%%%%%%%%%%%%%%%%%%%%%%%%%%%%%%%%%%%%%%%%%%%%%%%%%%%
\begin{proposition}\label{P2.1}
Suppose that $\bLg$ is recurrent. Then any $\bm{V}\in\Sobl^{2, d}(\Rd\times\cS)$
which is bounded from below in $\Rd\times\cS$ and satisfies $\bLg \bm{V}\le 0$ in 
$\Rd\times\cS$ must be equal to a constant, that is, $\bm{V}=(c,c,\ldots,c)$.
\end{proposition}

\begin{proof}
With no loss of generality we may assume that $\bm{V}>0$.
We pick some $z\in \Rd$ and then we show that 
$\inf_{j\in\cS}\inf_{\Rd} V_j\ge \min_{j\in \cS}V(z, j)$. 
Let $i\in\cS$ be such that $\min_{j\in \cS}V(z, j)= V(z, i)$.
Given $\delta\in (0, V(z, i))$, fix $\varepsilon>0$ small enough so that
$V(x, j)> V(z, i)-\delta$ for all $x\in B_\varepsilon(z)$ and $j\in\cS$.
Consider the sequence of solutions
$\bm{w}_n$ satisfying
\begin{align*}
\bLg \bm{w}_n &\,=\,0 \quad \text{in\ } 
(B_n(z)\setminus B_\varepsilon(z))\times\cS,
\\
\bm{w}_n &\,=\,0 \quad \text{on\ } \partial B_n(z)\times\cS,
\\
\bm{w}_n&\,=\, V(z, i)-\delta\quad \text{on\ } \partial B(z, \epsilon)\times \{i\}\,.
\end{align*}
Applying the maximum principle \cite[Theorem~3]{Sirakov}, it is easy to see that
\begin{equation}\label{EP2.1A}
0\le \bm{w}_n(x, j)\le \min\,\Bigl\{\max_j\,\max_{B_\varepsilon(z)}\,\bm{V}, V(x, j)\Bigr\}
\quad 
\text{in\ }
\bigl(B_n(z)\setminus B_\varepsilon(z)\bigr)\times\cS.
\end{equation}
Letting $n\to\infty$ and using standard elliptic estimates we find a bounded solution
$\bm{w}$ of 
\begin{align*}
\bLg \bm{w} =0 \quad \text{in\ } \Bar{B}^c_\varepsilon(z)\times\cS,\quad
\bm{w} = V(z, i)-\delta\quad \text{on\ } \partial B(z, \epsilon)\times \{i\}\,.
\end{align*}
Using the recurrence of $\bLg$ it is evident that $\bm{w}=V(z, i)-\delta$
and then, using \eqref{EP2.1A} we obtain $\bm{V}(z, i)-\delta\le \bm{V}(x, j)$
for all $x\in B^c_\varepsilon(z)$ and $j\in\cS$. Now letting $\delta, \varepsilon\to 0$
gives us
$$\inf_{j\in\cS}\,\inf_{\Rd}\, V_j \,\ge\, V(z, i) \,=\, \min_{j\in \cS}\,V(z, j)\,.$$
This of course, implies that $\bm{V}$ attains its minimum in $\Rd\times\cS$.
Let $\xi(x, j)=V(x, j)-V(z, i)$. Then $\bm{\xi}\ge 0$ and
$$\trace\bigl(a_i(x)\grad^2 \xi_{i}(x)\bigr)
+ b_i(x)\cdot \grad \xi_{i}(x) + m_{ii}(x)\xi_{i}(x)
\le (\bLg\bm{\xi})_i \,\le\, 0\quad \text{in\ } \Rd\,.$$
Since $m_{ii}\le 0$, by the strong maximum principle, this implies $\xi_i=0$ in $\Rd$.
This also implies that
\begin{equation}\label{EP2.1B}
0\,=\, (\bLg\bm{\xi})_j \,=\, \sum_{k\neq j} m_{jk}(x) \xi_k(x)\quad \text{in\ } \Rd\,.
\end{equation}
Using \eqref{irred} we find $k\in\cS\setminus\{i\}$ so that $m_{ik}(y)>0$ for some 
$y\in\Rd$. Hence from \eqref{EP2.1B} we get $\xi_k(y)=0$. Then repeating the above 
argument once again we have $\xi_k=0$ in $\Rd$. Now we can repeat the same argument
with the help of \eqref{irred} to arrive at $\bm{\xi}=0$ in $\Rd\times\cS$. This completes 
the proof.
\end{proof}

We also need a maximum principle which is a mild extension of \cite[Theorem~1]{Sirakov}. 
Consider a collection of smooth bounded domains $\{D_i\}$ with the property that $D_i\subset 
D$ for $1\le i\le N$. Let $g_i:\overline{D}\setminus D_i\to \RR$ be 
given continuous functions for $1\le i\le N$, and
$$G\,\df\,\max_i\, \max_{\bar{D}\setminus D_i}\, g^+_i\,.$$

%%%%%%%%%%%%%%%%%%%%%%%%%%%%%%%%%%%%%%%%%%%%%%%%%%%%%%%%%%%%%%%%%%%%%%%%%%%%%%%%%%%%
\begin{lemma}\label{L2.1}
Let $D_i\subset D\subset\Rd$, $i\in\cS$, be bounded domains, and $\bm{c}\le 0$.
Suppose that $u_i\in \Sobl^{2,d}(D_i)\cap \cC(\overline{D})$
satisfy 
$$(\bLg \bm{u})_i + c_i u_i\,\ge\, -f_i^+\quad \text{in\ } D_i, \quad
u_i=g_i\quad \text{in\ } D\setminus D_i,\quad \text{for all\ }i\in\cS\,,$$
with $\bm{c}\le 0$
Then for some constant $C$, not dependent on $\bm{u}, \bm{f}$ and $\bm{g}$, we have
\begin{equation}\label{EL2.1A}
\max_{i}\, \sup_{D_i} u^+_i \, \le\, \left(G + C\,\sum_{i=1}^d 
\norm{f^+_i}_{L^d(D_i)}\right)\,.
\end{equation}
Furthermore, if $f_i\in L^d(D_i)$ for $1\le i\le N$, then there exists a unique solution 
to 
\begin{equation}\label{EL2.1B}
(\bLg \bm{u})_i + c_i u_i= f_i\quad \text{in\ } D_i, \quad
u_i=g_i\quad \text{in\ } D\setminus D_i,\quad \text{for all\ } i\in\cS\,.
\end{equation}
\end{lemma}

\begin{proof}
To establish \eqref{EL2.1A} we follow the idea of \cite{Sirakov}. Let $j$ be such that
$\max_{i}\, \sup_{D_i} u^+_i = \sup_{D_j} u^+_j$. Replacing $u_i$ by 
$u_i-G$ we may assume that $G=0$.
Let $c_{jj} \df c_j + m_{jj}$. Since the 
equations are cooperative we have $c_{jj}\le -\sum_{k\neq j} m_{jk}\le 0$ in $D_j$.
Therefore, if $c_{jj}=0$ in $D_j$, then $\sum_{k\neq j} m_{jk}=0$ in $D_j$ which in turn,
makes the $j$ equation a scalar equation. Then we can apply the standard ABP estimate
to obtain \eqref{EL2.1A}. Thus, we assume that $c_{jj}\lneq 0$ in $D_j$.
Recall from \cref{E-Lk} that
$$L_j g \,=\, \trace\bigl(a_j(x)\grad^2 g(x)\bigr)
+ b_j(x)\cdot \grad g(x).$$
Let $v, w\in \Sobl^{2, d}(D_j)\cap \cC(\overline{D}_j)$ be such that
$$ L_j w \,=\, - f^+_j \quad \text{in\ } D_j, \quad w=0\quad \text{on\ } \partial D_j\,,$$
and
$$L_j v + c_{jj} v \,=\, c_{jj}  \quad \text{in\ } D_j, \quad v = 0
\quad \text{on\ } \partial D_j\,.$$
Applying \cite[Lemma~2.1]{Sirakov} we find $\delta>0$, dependent on $D_j$ and
the coefficients of $\bLg$, satisfying
$$0 \,\le\, v \,\le\, 1-\delta \quad \text{in\ } D_j\,.$$
Let $M=\max_i \sup_{D} u^+_i=\max_i \sup_{D_i} u^+_i$ (otherwise, there is nothing to prove).
We observe that 
\begin{align*}
L_j u_j + c_{jj} u_j &\,\ge\, - f^+_j - \sum_{k\neq j} m_{jk} u_k
\\
&\,\ge\, - f^+_j - \sum_{k\neq j} m_{jk} u^+_k
\\
&\,\ge\, - f^+_j + M c_{jj} \quad \text{in\ } D_j\,.
\end{align*}
Again, for $h=w + M v$, we have $L_j h + c_{jj} h \le -f^+_j + M c_{jj}$ in $D_j$.
Thus, by the strong maximum principle, we get $u_j\le h$ in $D_j$ giving us
$$M \,=\, \sup_{D_j}\, u_j\,\le\, \sup_{D_j}\, h \,\le\, 
C\norm{f^+_j}_{L^d(D_j)} + M (1-\delta)\,,$$
where we used the ABP estimate for $w$. This gives us \eqref{EL2.1A}.

Now that we have established the maximum principle, existence of a unique solution follows 
from a fixed point theorem. In particular, it is enough to prove existence
of a solution assuming $g_i\in \cC^2(\overline{D})$ for all $i$. For continuous $g$ we can
use a standard approximation argument. Now, replacing $u_i$ by $u_i-g_i$ we may assume that 
$g_i=0$ for all $i\in\cS$. Consider the set
$$\sH\,\df\,
\bigl\{\bm{v}\in\cC(\bar{D}\times\cS)\,\colon v_i=0 \text{\ in\ } \bar{D}\setminus D_i\,,
\text{\ for\ } i\in\cS\bigr\}\,.$$
For $\bm{v}\in\sH$ we define the map $\bm{w}=T\bm{v}\in \sH$ as follows:
$w_i\in\Sob^{2, d}(D_i)\cap\Sob^{1,d}_0(D_i)$ solves
$$ L_i w_i + c_{ii} w_i = -f_i - \sum_{k\neq i} m_{ki} v_k \quad \text{in\ } D_i\,,
\quad\text{and}
\quad w_i=0\quad \text{on\ } \partial D_i.$$
In view of \cite[Theorem~9.15]{GilTru}, $T$ is well defined. Further more $T$ is 
linear, continuous and compact. Now setting $\sK\subset \sH$ as the collection of
$\beta$-H\"{o}lder continuous functions for some small $\beta$, we note from
\cite[Corollary~9.29]{GilTru} that $T:\sK\to\sK$. Since $\sK$ is a compact, convex
subset of $\sH$, using Schauder fixed point theorem we get a fixed point $\bm{u}$ of $T$. 
It is easily seen that $\bm{u}$ is the solution of \eqref{EL2.1B}.
\end{proof}

We are now ready to prove \cref{T1.1}

%%%%%%%%%%%%%%%%%%%%%%%%%%%%%%%%%%%%%%%%%%%%%%%%%%%%%%%%%%%%%%%%%%%%%%%%%%%%%%%%%%%%%
\begin{proof}[{\bf Proof of Theorem~\ref{T1.1}}]
First we consider (a). Suppose, on the contrary, that there exists a bounded, positive
$\bm{u}\in\Sobl^{2, d}(\Rd\times\cS)$ solving
$\bLg \bm{u} = \bm{C}\, \bm{u}$ for some constant vector $\bm{C}>0$. Since $\bLg \bm{u}\ge 
0$ in $\Rd\times\cS$. 
From the proof of \cref{P2.1} it follows that $\bm{u}$ attains its maximum in $\Rd\times\cS$,
say in the component $u_i$, and $u_i$ is constant in $\Rd$. Then
$$0 \,<\, C_i u_i \,=\, (\bLg \bm{u})_i \,=\, \sum_{k\neq i} m_{ik} (u_j-u_i) \,\le\, 0\,,$$
which is a contradiction.
This proves (a).

Next we consider (b). Fix a ball $B\subset\Rd$ and $\cS_1\subset \cS$. Given 
a continuous function $\bm{g}: \Bar{B}\times \cS_1\to \RR$, we first prove
existence of solution to 
\begin{equation}\label{ET1.1A}
\bLg \bm{u} \,=\, 0 \quad \text{in\ } (\Bar{B}\times\cS_1)^c, 
\quad \bm{u}\,=\, \bm{g} \quad \text{on\ } \Bar{B}\times\cS_1.
\end{equation}
Applying \cref{L2.1}, we consider a sequence of solutions satisfying
\begin{equation}\label{ET1.1B}
\begin{split}
\bLg \bm{u}_n &\,=\,0 \quad \text{in\ } (B_n\times\cS)\setminus (B\times \cS_1),
\\
\bm{u}_n &\,=\,\bm{g} \quad \text{on\ } \bar{B}\times\cS_1\,,
\\
\bm{u}_n &\,=\,0 \quad \text{on\ } \partial{B_n}\times\cS\,.
\end{split}
\end{equation}
One more application of \cref{L2.1} gives 
$$-\max_i\,\max_{\Bar B}\,|g_i| \,\le\, u_n(x,i) \,\le\, \max_i\,\max_{\Bar B}\,|g_i|\,,
\quad (x,i)\in (B_n\times\cS)\setminus (B\times \cS_1).$$
Thus, using standard elliptic estimates we can pass to the limit in \eqref{ET1.1B}
to find a solution $\bm{u}$ of \eqref{ET1.1A}.

Next we divide the proof of uniqueness in three steps.

\noindent{\bf Step 1.} Let $K\Subset B$ and $\cS_1=\cS$ where
$K$ is from \eqref{ED1.2A}. 
%To prove uniqueness of $\bm{u}$ in \eqref{ET1.1A} it is 
%enough to show that for $\bm{g}=0$ and have $\bm{u}=0$. So we assume $\bm{g}=0$.
We claim that for a solution $\bm{u}$ of \eqref{ET1.1A} we have
\begin{equation}\label{ET1.1C}
\sup_j\,\sup_{B^c}\, u_j\,=\, \max_{j}\, \max_{\partial{B}}\, u_j\,.
\end{equation}
Replacing $u_i$ by $u_i- \max_{j}\, \max_{\partial{B}} u_j$ we may assume that the
rhs of \eqref{ET1.1C} is $0$.
Suppose, on the contrary, that the claim \eqref{ET1.1C} is not true. 
Then we must have 
$$0< \max_{j}\,\sup_{B^c}\, u_j\,<\, \infty\,.$$
Since $\bm\Lyap\ge 1$ is bounded below, without any loss of generality we may assume 
$\bm\Lyap>\bm{u}$. Otherwise, multiply $\bm{u}$ with a suitable positive constant.
Now we choose $n$ large enough
so that $(\bm{u}-\frac{1}{n}\bm\Lyap)(x_0, j_0)>0$ for some $(x_0, j_0)\in B^c\times\cS$. 
Again, we choose $\kappa>0$ small enough so that
for $\bm\phi=(\bm{u}-\frac{1}{n}\bm\Lyap)$ we have
$$\bLg\bm\phi\,=\,-\frac{1}{n}\bLg\bm\Lyap \,\ge\, \frac{\kappa_1}{n}\bm\Lyap
>\kappa \bm\phi \quad \text{in\ } B^c\times\cS\,,$$
where the inequality follows from \eqref{ED1.2A}.
Let $\bm{w}_n$ be the solution of 
\begin{equation}\label{ET1.1D}
\bLg \bm{w}_n=\kappa \bm{w}_n\quad \text{in\ } B_n\times\cS, \quad \text{and}\quad  
\bm{w}_n=\norm{\bm\phi^+}_{L^\infty}\quad 
\text{on\ } \partial B_n\times\cS\,.
\end{equation}
Using the scalar maximum principle 
we have $\bm{w}_n>0$ in $B_n\times\cS$. Indeed, 
if $\min_{j}\min_{\Bar B_n} (\bm{w}_n)_j = (\bm{w}_n)_i(z)\le 0$, then using \eqref{E-Lk} we write
\begin{align*}
L_i \bigl((\bm{w}_n)_i-(\bm{w}_n)_i(z)\bigr) &+ (m_{ii}-\kappa)
\bigl((\bm{w}_n)_i-(\bm{w}_n)_i(z)\bigr)\\
&\,\le\, \Bigl(\bLg \bigl(\bm{w}_n-(\bm{w}_n)_i(z)\bigr)\Bigr)_i
-\kappa \bigl((\bm{w}_n)_i-(\bm{w}_n)_i(z)\bigr)
\,\le\, 0\,,
\end{align*}
and therefore, by the strong maximum principle, we must have $(\bm{w}_n)_i=(\bm{w}_n)_i(z)$ in $B_n$ which is not possible since $(\bm{w}_n)_i>0$ on $\partial B_n$.
Since  $\bLg\bm\phi -\kappa \bm\phi\ge 0$ in $(B_n\cap B^c)\times\cS$ and $\bm\phi\le 0$ in $\partial B\times\cS$, 
using \cref{L2.1} it follows that $\bm{w}_n\ge \bm\phi$ in $(B_n\cap B^c)\times\cS$.
Furthermore, since $\bm{w}_n$ attains its maximum at the boundary,
we have $\sup_j\sup_{B_n}(\bm{w}_n)_j\le \norm{\bm{\phi}^+}$. 
Thus, passing to the limit in \eqref{ET1.1D}, we find a solution $\bm{v}$ of 
$\bLg\bm{v}=\kappa \bm{v}$ which is bounded and non-negative.
Again, $\bm{v}(x_0, j_0)\ge \bm\phi(x_0, j_0)$ implies that 
$\bm{v}$ is positive, due to the maximum principle. This of course, contradicts regularity
of $\bLg$.
Hence, we must have 
$$\sup_j\,\sup_{B^c}\, u_j \,\le\, 0\,,$$
which establishes the claim \eqref{ET1.1C}. Now using \eqref{ET1.1C} we can easily obtain
uniqueness of solution \eqref{ET1.1A} when $B\Supset K$ and $\cS_1=\cS$.

\noindent{\bf Step 2.} We show that if the exterior problem \eqref{ET1.1A}
with respect to a set $B'\times\cS'$  has a unique bounded
solution with boundary data $0$, then the same is the case for the 
exterior problem with respect to any domain 
$B^{\prime\prime}\times\cS^{\prime\prime}$ where $B'\subset B^{\prime\prime}$
and $\cS'\subset\cS^{\prime\prime}\subset\cS$. 
Suppose, on the contrary, that there exists 
$B^{\prime\prime}\times\cS^{\prime\prime}$ such that \eqref{ET1.1A} has a non-zero
bounded solution 
$\bm{v}$ with 
boundary data $0$ given in $\Bar{B}^{\prime\prime}\times\cS^{\prime\prime}$. With no loss 
of generality, we may 
assume that $\bm{v}^+\gneq 0$.
Now consider the sequence of solutions $\bm{v}_n$ of
\begin{equation}\label{ET1.1E}
\begin{split}
\bLg \bm{v}_n &\,=\,0 \quad \text{in\ }  (B_n\times\cS)\setminus (B'\times\cS')\,,
\\
\bm{v}_n &\,=\,0\quad \text{on\ } \Bar{B}'\times\cS',
\quad \bm{v}_n\,=\,\norm{\bm{v}^+}_{L^\infty}\quad \text{on\ } \partial B_n\times\cS\,. 
\end{split}
\end{equation}
Furthermore, $0\le \bm{v}_n\le \norm{\bm{v}^+}_{L^\infty}$, 
by an argument similar to \eqref{ET1.1D}.
Again, by the comparison
principle in \cref{L2.1}, we get $\bm{v}^+\le \bm{v}_n\le \norm{\bm{v}^+}_{L^\infty}$
in $(B^{\prime\prime}\times\cS^{\prime\prime})^c$.
Therefore, using standard elliptic pde estimates we can pass to
the limit in \eqref{ET1.1E}, as $n\to\infty$, to obtain a solution $\bm{u}$ satisfying
$$\bLg\bm{u} \,=\, 0 \quad \text{in\ }  (B'\times\cS')^c,
\quad \bm{u} = 0\quad \text{on\ } \Bar{B}'\times\cS',$$
and $\bm{u}\ge \bm{v}^+$. But this contradicts the uniqueness hypothesis with respect to 
the domain $B'\times\cS'$.

\noindent{\bf Step 3.} In view of Step 2, it is enough to prove uniqueness of 
\eqref{ET1.1A} with respect to domains of the form $B\times\{i\}$. Again, we may choose
$|B|$ small enough so that $\bm{M}$ is irreducible (see \eqref{irred}) in 
$\Rd\setminus B$. Now consider a solution $\bm{u}$ of the problem
$$\bLg\bm{u} \,=\, 0 \quad \text{in\ } (B\times\{i\})^c,
\quad \bm{u} = 0\quad \text{on\ } \Bar{B}\times\{i\}\,.$$
We have to show that $\bm{u}=0$. Choose $\widehat{B}\Supset K\cap B$. Then, from 
\eqref{ET1.1C}, we get 
\begin{equation*}
\sup_j\,\sup_{\widehat{B}^c}\, u_j\,=\, \max_{j}\, \max_{\partial{\widehat B}}\, u_j\,.
\end{equation*}
This of course, implies that $\bm{u}$ attends its maximum in $\widehat{B}\times\cS$.
Using irreducibility, it is now easy to show that $\bm{u}\le 0$ (see the argument in
\cref{P2.1}). Likewise, we can also show that $\bm{u}\ge 0$.

Part (c) can be treated as a special case, using the arguments in part (b).
This completes the proof.
\end{proof}

Let us also include the following useful characterization of recurrence.
A similar result can be found in \cite[Theorem~3.12]{YZ10}
in a more restrictive setting.

%%%%%%%%%%%%%%%%%%%%%%%%%%%%%%%%%%%%%%%%%%%%%%%%%%%%%%%%%%%%%%%%%%%%%%%%%%%%%%%%%%%%%
\begin{proposition}\label{P2.2}
Suppose that for some ball $B_r(x_0)$
and $\cS_1\subset\cS$ the exterior
Dirichlet problem 
$$\bLg \bm{u} \,=\, 0\quad  \text{in\ } 
\bigl(\Bar{B}_r(x_0)\times\cS_1\bigr)^c\,,$$
with given continuous
boundary values on $\Bar{B}_r(x_0)\times\cS_1$ has a unique
bounded solution.
Then $\bLg$ is recurrent.
\end{proposition}

\begin{proof}
Without any loss of generality, we may assume that $B_r(x_0)\times\cS_1=B_1(0)\times\cS_1$.
For some ball $B$ consider the set $B\times\cS_2$. As shown in the 
proof of \cref{T1.1}, given a function
$\bm{g}\colon \Bar{B}\times \cS_2\to \RR$, there exists
a bounded solution to 
\begin{equation}\label{EP2.2A}
\bLg \bm{u} \,=\, 0 \quad \text{in\ } (\Bar{B}\times\cS_2)^c, 
\quad \bm{u}\,=\, \bm{g} \quad \text{on\ } \Bar{B}\times\cS_2.
\end{equation}
Thus we only need to establish the uniqueness of \eqref{EP2.2A},
in other words, that $\bm{g}=0$ implies
$\bm{u}=0$. 

First, consider the case when $B_1(0)\times\cS_1\subset B\times\cS_2$. Suppose, on the contrary, that $\bm{u}\neq 0$. Then repeating the argument of step 2 in \cref{T1.1} we can construct a solution $\bm{v}$ to
$$\bLg\bm{v} \,=\, 0 \quad \text{in\ } (B_1(0)\times\cS_2)^c,
\quad \bm{u} \,=\, 0\quad \text{on\ } \Bar{B}_1(0)\times\cS_2,$$
and $\bm{v}\ge \bm{u}^+$. This clearly, contradicts the hypothesis
of the proposition. Thus 
$\bm{u}= 0$.

Next, we examine the case where 
$B\times\cS_2\subset B_1(0)\times\cS_1$.
We claim that for any solution $\bm{v}$ of 
$$\bLg\bm{v} \,=\, 0 \quad \text{in\ } (B\times\cS)^c,$$
we have
\begin{equation}\label{EP2.2B}
\max_j\, \max_{\Bar B}\, v_j\,=\, \max_j \sup_{B^c}\, v_j\,.
\end{equation}
\eqref{EP2.2B} follows from the uniqueness of solution, comparison principle in bounded domains
and approximation of $v$  by a sequence of solution as done in \eqref{ET1.1B}.
Now using \eqref{EP2.2B} we can complete the proof of uniqueness repeating an argument similar to
step 3 of \cref{T1.1}.
\end{proof}

The following observation is used in several places.

%%%%%%%%%%%%%%%%%%%%%%%%%%%%%%%%%%%%%%%%%%%%%%%%%%%%%%%%%%%%%%%%%%%%%%%%%%%%%%%%%%%%%
\begin{lemma}\label{L2.2}
Suppose that $(\bPsi, \bm c, \lambda)$ and
$(\widehat\bPsi, \widehat{\bm c}, \widehat\lambda)$ be two tuples satisfying
\begin{equation*}
\bLg\bPsi + \bm{c}\bPsi \,=\, -\lambda \bPsi\,,\qquad\text{and\ \ }
\bLg\widehat\bPsi + \widehat{\bm{c}}\widehat\bPsi \,=\, -\widehat\lambda \widehat\bPsi
\end{equation*}
on $\Rd$.
Define ${\Phi}_k(x)\df\frac{\widehat\Psi_k}{\Psi_k}(x)$,
and $\bPhi \df (\Phi_1,\dotsc,\Phi_N)$.
Then, the following identity holds:
\begin{equation*}
\widetilde{\bLg}^{\bm\psi} \boldsymbol\Phi
+ \bigl(\widehat{\bm{c}} - \bm{c} - (\lambda-\widehat\lambda)\bigr)\bPhi\,=\,0\,.
\end{equation*}
Moreover, for $\widehat\Psi_k=\Phi_k\Psi_k$ we have
\begin{equation*}
(\bLg\widehat\bPsi)_k \,=\, \Phi_k (\bLg\bPsi)_k + \Psi_k(\widetilde\bLg^{\bm\psi}\bPhi)_k
\quad\forall\,k\in\cS\,.
\end{equation*}
\end{lemma}

\begin{proof}
Note that the first identity follows from the second one,
which can be shown as follows:
\begin{align*}
(\bLg\widehat\bPsi)_k
&\,=\, \trace\bigl(a_k(x)\grad^2 \widehat\Psi_{k}(x)\bigr)
+ b_k(x)\cdot \grad \widehat\Psi_{k}(x) + \sum_{j\neq k} 
m_{kj}(x)(\widehat\Psi_{j}(x)-\widehat\Psi_k(x))
\\
&\,=\, \Phi_k(x) \Bigl[\trace\bigl(a_k(x)\grad^2 \Psi_{k}(x)\bigr)
+ b_k(x)\cdot \grad \Psi_{k}(x) + \sum_{j\neq k} m_{kj}(x)(\Psi_{j}(x)-\Psi_k(x)) \Bigr]
\\
&\mspace{50mu} + \Psi_k(x) \Bigl[\trace\bigl(a_k(x)\grad^2 \Phi_{k}(x)\bigr)
+ (b_k(x) + 2 a_k(x)\grad\psi_k(x))\cdot\grad \Phi_{k}(x) \Bigr]
\\
&\mspace{100mu} + \sum_{j\neq k} m_{kj}(x)(\Phi_{j}(x)-\Phi_k(x))\Psi_j(x)
\\
&\,=\, \Phi_k(x) (\bLg\bPsi)_k + \Psi_k(x) L^{\bm\psi}_k \Phi_k + \Psi_k(x)
\sum_{j\neq k} \widetilde{m}_{kj}(x)(\Phi_{j}(x)-\Phi_k(x))
\\
&\,=\, \Phi_k (\bLg\bPsi)_k + \Psi_k(\widetilde\bLg^{\bm\psi}\bPhi)_k\,.
\end{align*}
This completes the proof.
\end{proof}

Next we prove \cref{T1.3}.

%%%%%%%%%%%%%%%%%%%%%%%%%%%%%%%%%%%%%%%%%%%%%%%%%%%%%%%%%%%%%%%%%%%%%%%%%%%%%%%%%%%%
\begin{proof}[{\bf Proof of Theorem~\ref{T1.3}}]
Let $(\bPsi_D, \lambda_D)$ be the Dirichlet principal eigenpair solving \eqref{E1.12}. 
Consider a 
closed ball $\sB\Subset D$ and let $K=\sB\times \cS$. Applying \cref{TA.2} it follows that
$\lambda_D(\bm{c}-\Ind_K)>\lambda_D(\bm{c})=\lambda_D$.
Again, using \cref{TA.3,TA.4} we can find
a smooth domain $D_1$ which contains $\Bar{D}$ and 
$\lambda_1\df\lambda_{D_1}(\bm{c}-\Ind_K)>\lambda_D$.

Let $\bPsi_{D_1}$ be the Dirichlet principal eigenfunction corresponding to the eigenvalue
$\lambda_{D_1}(\bm{c}-\Ind_K)$ in the domain $D_1$. That is,
$\bPsi_{D_1}\in \cC_{0}(\overline{D}_1\times\cS)
\cap\Sobl^{2,p}(D_1\times\cS),\ p>d$, and
\begin{equation}\label{ET1.3B}
\begin{split}
\bLg \bPsi_{D_1} + (\bm{c}-\Ind_{K})\bPsi_{D_1} &\,=\, -\lambda_1\, \bPsi_{D_1} \quad 
\mbox{in\ }
D_1\times\cS\,,\\
\bPsi_{D_1} &\,=\, 0\quad \text{on\ } \partial{D_1}\times\cS\,,\\
\bPsi_{D_1} &\, > \, 0 \quad \text{in\ } D_1\times\cS\,.
\end{split}
\end{equation}
Now define
$$(\bm\Lyap_D)_k \,\df\,
\frac{(\bPsi_{D_1})_k}{(\bPsi_D)_k}\quad \text{for\ } k\in\cS.$$
Since $\bPsi_D=0$ on $\partial{D}\times\cS$, it is evident that $\bm\Lyap_D$ is inf-compact 
and
$\bm\Lyap_D\in \Sobl^{2,p}(D\times\cS)$ for any $p>d$. Using \cref{L2.2} and \eqref{ET1.3B} 
we then obtain
\begin{equation}\label{ET1.2C}
\widetilde\bLg_D\bm\Lyap_D \,=\, (\lambda_D-\lambda_1+\Ind_{K})\bm\Lyap_D
\,\le\, -\delta_1 \bm\Lyap_D + \delta_2 \Ind_{K}\,,
\end{equation}
where $\delta_1=\lambda_1-\lambda_D$ and $\delta_2=\max_k\max_{\sB}(\bm\Lyap_D)_k$. This proves 
\eqref{ET1.3A}.

Now we show that the twisted operator $\widetilde{\bLg}_D$ is regular. Suppose,
on the contrary, that it is not. Then we can find a positive vector
$\bm{C}$ and a bounded, positive solution $\bm{u}$ of
$$\widetilde{\bLg}_D\bm{u}\,=\, \bm{C}\bm{u}\quad \text{in\ } D\times\cS.$$
Define $\Phi_k\df(\bPsi_D)_k u_k$. Using \cref{L2.2}, we then have
$$\bLg\bPhi \,=\, \bm{u}\bLg\bPsi_D + \bPsi_D\widetilde\bLg_D\bm{u}
\,=\, -\lambda_D\bPhi + \bm{C} \bPhi\,\ge\, (-\lambda_D + C_0) \bPhi$$
in $D\times\cS$, where $C_0=\min\{C_1, \ldots, C_N\}>0$. But this is not possible due to
\eqref{ELA.3A}. Hence $\widetilde{\bLg}_D$ is regular, completing the proof.
\end{proof}

%%%%%%%%%%%%%%%%%%%%%%%%%%%%%%%%%%%%%%%%%%%%%%%%%%%%%%%%%%%%%%%%%%%%%%%%%%%%%%%%%%%%
Next we consider the eigenvalue problem for $\bcA$ in $\Rd$.
Using \eqref{irred}, we can find a ball $B_{n_0}$ such that the matrix
$\bm{M}$ is irreducible in $B_{n_0}$. Then, by \cref{TA.1}, there exists
$n_0\in\NN$, and
a unique pair 
$(\bPsi_n,\lambda_n)\in 
\cC_{0}(\Bar{B}_{n}\times\cS)\cap\Sobl^{2,p}(B_{n}\times\cS)\times\RR$,
$p>d$,
satisfying
\begin{equation}\label{E2.12}
\begin{split}
\bm\cA \bPsi_{n} &\,=\, -\lambda_n\, \bPsi_{n} \quad \mbox{in\ } B_n\times\cS\,,\\
\bPsi_{n} &\,=\, 0\quad \text{on\ } \partial{B_n}\times\cS\,,\\
\bPsi_{n} &\, > \, 0 \quad \text{in\ } B_n\times\cS\,,
\end{split}
\end{equation}
for all $n\ge n_0$. The uniqueness of $\bPsi_n$ holds up to a multiplicative constant.
Furthermore, by \cref{TA.3} and \eqref{EA.1}, we have
$\lambda_n>\lambda_{n+1}\ge \lamstr$.
Hence, it suffices show that the Dirichlet principal eigenvalues $\{\lambda_n\}$
form a monotone sequence that tends to the principal value $\lamstr$ as $n\to\infty$.
But this is a simple generalization of \cite[Lemma~2.2]{ABS19} to systems.

%%%%%%%%%%%%%%%%%%%%%%%%%%%%%%%%%%%%%%%%%%%%%%%%%%%%%%%%%%%%%%%%%%%%%%%%%%%%%%%%%%%%
\begin{lemma}\label{L2.3}
Suppose that $\Tilde{\lambda} = \lim_{n\to\infty}\lambda_n> -\infty$.
Then the following hold:
\begin{itemize}
\item[(a)]
There exists a function
$\bPsi^*\in \Sobl^{2,p}(\Rd\times\cS)$, $\bPsi^* > 0$, satisfying
\begin{equation}\label{EL2.2A}
\bm\cA\bPsi^* \,=\, -\Tilde{\lambda} \bPsi^*\quad \text{in\ } \Rd\times\cS\,.
\end{equation}

\item[(b)]
It holds that $\Tilde{\lambda} = \lamstr$.

\item[(c)]
$\lamstr$ is concave in $c$.
\end{itemize}
\end{lemma}

\begin{proof}
By \eqref{E2.12}, for each $n\ge n_0$, the function
$\bPsi_n=(\Psi_{n,1},\dotsc,\Psi_{n,N})$ satisfies 
\begin{equation*}
(\bcA\bPsi_{n})_k(x) \,=\, -\lambda_{n}(\bPsi_n)_{k}(x)\,,
\quad x\in B_n\,,\quad k\in\cS\,.
\end{equation*}
Let $\sK\subset B_n$ be any compact set, and without loss of generality,
assume that $0\in\sK$.
Scale $\bPsi_n$, so that $\min\bigl\{\Psi_{n,1}(0), \dotsc, \Psi_{n, N}(0)\bigr\}=1$.
Applying Harnack's inequality \cite[Theorem~2]{Sirakov} (see also \cite{AGM93,BS04}), we 
obtain
\begin{equation*}
\sup_{y\in\sK}\,\max_{k\in\cS}\,\Psi_{n,k}(y) \,\le\, C_{\mathrm{H}} \,,
\end{equation*}
for some constant $C_{\mathrm{H}}$ independent of $n$.
Thus, by \cite[Theorem~9.11]{GilTru}, it follows that for any domain $Q\subset \sK$
and any $p>d$, there exists a constant $\kappa_1$ such that
\begin{equation*}
\bnorm{\bPsi_n}_{\Sob^{2,p}(Q\times\cS)} \,\le\, \kappa_{1}\quad\forall\,n\ge n_0\,.
\end{equation*}
Hence, by a standard diagonalization argument, we can extract a subsequence
$\{\bPsi_{n_k}\}$ such that
\begin{equation*}
\bPsi_{n_k}\to \bPsi^*\quad \text{in\ } \Sobl^{2,p}(\Rd\times\cS)\,, \quad
\text{and}\quad \bPsi_{n_k}\to \bPsi^*\quad \text{in\ }
\cC^{1, \alpha}_{\text{loc}}(\Rd\times\cS)
\end{equation*}
for some $\bPsi^*\in \Sobl^{2,p}(\Rd\times\cS)$.
Moreover, we have
\begin{equation}\label{EL2.3B}
\bm\cA\bPsi^* \,=\, -\Tilde{\lambda}\bPsi^*\quad\text{in}\,\, \Rd\,.
\end{equation}
Since $\min_{k\in\cS}\Psi_{k}^*(0) \ge 1$, another application of Harnack's
inequality shows that $\bPsi^* > 0$ in $\Rd\times\cS$.
This gives us \eqref{EL2.2A} and hence the proof of part (a) is complete.

We continue with part (b).
It is clear from \cref{EL2.3B} and 
\cref{E1.9} that $\Tilde{\lambda} \le \lamstr$.
Again, from the definition in \cref{EA.1}, we have $\lambda_{n} \ge \lamstr$
for all $n\in\NN$. This implies that $\Tilde{\lambda} \ge \lamstr$.
Therefore, we obtain $\Tilde{\lambda} = \lamstr$. This proves part (b).

Part (c) follows from \cref{LA.2} and (b).
\end{proof}

%%%%%%%%%%%%%%%%%%%%%%%%%%%%%%%%%%%%%%%%%%%%%%%%%%%%%%%%%%%%%%%%%%%%%%%%%%%%%%%%%%%%
\begin{remark}
Since $\tilde\lambda=-\infty$ implies $\lamstr=-\infty$,
it follows from the proof of \cref{L2.3} that $\Tilde\lambda=\lamstr\in(-\infty, \infty]$. 
\end{remark}

Now we can prove \cref{T1.2}

%%%%%%%%%%%%%%%%%%%%%%%%%%%%%%%%%%%%%%%%%%%%%%%%%%%%%%%%%%%%%%%%%%%%%%%%%%%%%%%%%%%%
\begin{proof}[\bf Proof of Theorem~\ref{T1.2}]
In view of \cref{L2.3} we only need to consider the case $\lambda < \lamstr$.
It follows from \cref{TA.3} that for any bounded domain
$D$ we have $\lambda_D(\bm\cA + \lambda) > 0$ where $\lambda_D(\bm\cA + \lambda)$
is the Dirichlet principal eigenvalue of $\bm\cA+\lambda$ in $D$ (see \cref{A-Eigen}).
Let $n\ge n_0$.
Let $D_n\Subset B_n\setminus B_{n-1}$ be a ball of radius $\nicefrac{1}{4}$.
Then using \cref{LA.3} we can find a positive function $\widetilde{\bPsi}_n$
satisfying
$(\bm\cA+\lambda)\widetilde{\bPsi}_n = -\Ind_{D_n\times\cS}$ in $B_n\times\cS$.
We scale $\widetilde{\bPsi}_n$ so that
$\min\,\bigl\{\widetilde{\Psi}_{n,1}(0),\dotsc, \widetilde{\Psi}_{n, N}(0)\bigr\} = 1$.
Now using Harnack's inequality \cite{Sirakov}, it is easy to show that the functions
$\{\Tilde{\bPsi}_n\}$ are locally uniformly bounded, and therefore
by standard elliptic estimates they are locally uniformly bounded in $\Sob^{2,p}$.
Thus we can extract a subsequence and find an eigenfunction $\bPsi$ solving
$\bm\cA\bPsi + \lambda\bPsi=0$.
Again, we have $\bPsi> 0$ by the strong maximum principle.
This completes the proof.
\end{proof}

At this point we recall that for the remaining part we need \cref{H1.1}, which
is enforced without any further mention. Next we produce a proof of \cref{T1.4}

%%%%%%%%%%%%%%%%%%%%%%%%%%%%%%%%%%%%%%%%%%%%%%%%%%%%%%%%%%%%%%%%%%%%%%%%%%%%%%%%%%%%
\begin{proof}[\bf Proof of \cref{T1.4}]
Let $(\bPsi, \lambda)$ be an eigenpair with $\lambda<\lamstr$ and $\widetilde{\bLg}$
be the corresponding twisted operator. Suppose, on the contrary, that  
$\widetilde{\bLg}$ is recurrent. Let $\Phi_k=\frac{\Psi^*_k}{\Psi_k}$. Using \cref{L2.2} we 
then get
$$ \widetilde{\bLg}\bPhi \,=\, (\lambda-\lamstr)\bPhi \,<\, 0\quad \text{in\ } \Rd\times\cS\,.$$
Since $\widetilde{\bLg}$ is recurrent, applying \cref{P2.1} we see that $\bPhi$ is 
constant, i.e., for some $c>0$ we have $\bPhi=(c,c,\ldots, c)$.  This also gives us
$(\lambda-\lamstr)=0$ which is a contradiction. This completes the proof.
\end{proof}

Next we show that minimal growth at infinity is equivalent to the monotonicity property
on the right.

%%%%%%%%%%%%%%%%%%%%%%%%%%%%%%%%%%%%%%%%%%%%%%%%%%%%%%%%%%%%%%%%%%%%%%%%%%%%%%%%%%%%%
\begin{proof}[\bf Proof of \cref{T1.8}]
First we show that (b)$\,\Rightarrow\,$(a). Consider a
function $\bm{h}\in \sB^+_0(\Rd\times\cS)$. From
the definition of $\lamstr$ in \eqref{E1.9} it follows that 
$\lamstr(\bm{c}+\bm{h})\le \lamstr(\bm{c})$. Now suppose, to the contrary, that 
$\lamstr(\bm{c}+\bm{h})= \lamstr(\bm{c})$. Let $\widehat\bPsi$ be a principal eigenfunction 
corresponding to $\lamstr(\bm{c}+\bm{h})$. Then
\begin{equation}\label{ET1.8A}
\bLg\widehat\bPsi + (\bm{c}+\lamstr)\widehat\bPsi\,\le\,
\bLg\widehat\bPsi + (\bm{c}+\bm{h}+\lamstr)\widehat\bPsi\,=\, 0\quad \text{in\ }
\Rd\times\cS.
\end{equation}
Since $\bPsi^*$ has minimal growth property at infinity, it follows from the proof
of \cref{T1.7}  that $\widehat\bPsi= \kappa \bPsi^*$ for some $\kappa>0$.
From \eqref{ET1.8A} we then get $\bm{h}\widehat\bPsi=0$, which contradicts
to the fact that $\bm{h}\neq 0$. Hence we must have $\lamstr(\bm{c}+\bm{h})< 
\lamstr(\bm{c})$. Thus we get (a).

Next, we show that (a)$\,\Rightarrow\,$(b). We construct a principal eigenfunction $\bPsi^*$
with the minimal growth at infinity. Let $\bm{f}=\Ind_{B_1\times\cS}$ and
$\bPsi_n\in\Sobl^{2, p}(B_n\times\cS)\cap \cC(\Bar{B}_n\times\cS)$ be the unique solution 
of 
\begin{equation}\label{ET1.8B}
\begin{split}
\bLg\bPsi_n + (\bm{c}+\lamstr)\bPsi_n &\,=\, - \bm{f} \quad \text{in\ } B_n\times\cS,
\\
\bPsi_n &\,>\, 0 \quad \text{in\ } B_n\times\cS,
\\
\bPsi_n &\,=\, 0 \quad \text{on\ } \partial{B}_n\times\cS.
\end{split}
\end{equation}
Existence of $\bPsi_n$ follows from \cref{LA.3}. Let
$$\beta_n\,\df\,\max_{j\in\cS}\,\max_{\Bar B_1}\, (\bPsi_n)_j\,.$$
We claim that $\beta_n\to\infty$ as $n\to\infty$. Arguing
by contradiction, suppose that $\{\beta_{n_k}\}$ is bounded for some subsequence $\{n_k\}$.
Let
$$
\kappa_n\,\df\,\sup\,\{t\,\colon \bPsi^*-t\bPsi_n>0\; \text{in\ } \Bar{B}_1\times\cS\}\wedge 1\,,
$$
where $(\bPsi^*,\lamstr)$ denotes the principal eigenpair. Note that, by this
hypothesis, we have $\inf_{n_k} \kappa_{n_k}>0$. Letting $\bm{v}_n=\kappa_n\bPsi_n$,
we note that $\bm{v}_{n_k}\le \bPsi^*$ in $\Bar{B}_1\times\cS$, and 
$$
\bLg (\bPsi^*-\bm{v}_{n_k}) + (\bm{c}+\lamstr)(\bPsi^*-\bm{v}_{n_k}) \,=\,0\quad 
\text{in\ } 
(B_n\setminus \Bar{B}_1)\times\cS.
$$
Since $\lambda_{B_n\cap B^c_1}>0$,  it follows from \eqref{ELA.3A} that 
$\bm{v}_{n_k}\le \bPsi^*$ in $B_n\times\cS$. Now using standard elliptic pde
estimates we can find a subsequence of $\bm{v}_{n_k}$ converging to 
some non-negative $\bm{v}$, and using \eqref{ET1.8B} we have
$$\bLg\bm{v} + (\bm{c}+\lamstr)\bm{v} \,=\, - \kappa\bm{f}\,\le\, 0 \quad \text{in\ } 
\Rd\times\cS,$$
where $\kappa>0$ is obtained along some subsequential limit of $\{\kappa_{n_k}\}$.
By the strong maximum principle, either $\bm{v}=0$ or $\bm{v}>0$ in $\Rd$. The former 
is not possible as $f\neq 0$. So we must have $\bm{v}>0$. Letting 
$h_k=\kappa\frac{f_k}{v_k}$
we obtain from above that
$$\bLg\bm{v} + (\bm{c}+ \bm{h}+\lamstr)\bm{v} \,=\, 0 \quad \text{in\ } \Rd\times\cS\,.$$
From \eqref{E1.9} we then have $\lamstr(\bm{c}+ \bm{h})\ge \lamstr$ which contradicts
the hypothesis of monotonicity on the right. Therefore, we must have 
$\beta_n\to\infty$ as $n\to\infty$. In this case we have $\kappa_n\to 0$ as $n\to\infty$. 
Also, we note that for all large $n$ we have 
\begin{equation}\label{ET1.8C}
\min_{j}\,\min_{\Bar B_1}\, (\Psi^*_j-(\bm{v}_n)_j)\,=\,0\,.
\end{equation}
As before, we can pass to the limit in \eqref{ET1.8B} (after multiplying both sides by
$\kappa_n$) to obtain a non-negative solution $\bm{v}$ of 
$$\bLg\bm{v} + (\bm{c}+\lamstr)\bm{v} \,=\, 0 \quad \text{in\ } \Rd\times\cS\,.$$
In view of \eqref{ET1.8C}, we must have
\begin{equation*}
\min_{j}\,\min_{\Bar B_1}\, (\Psi^*_j-v_j) \,=\, 0\,,
\end{equation*} 
which means $\bm{v}>0$ in $\Rd\times\cS$.

To complete the proof, it is enough to show that $\bm{v}$ has minimal growth at infinity. 
Consider a positive $\bPhi\in\Sobl^{2,p}((K\times\cS_1)^c)$, $p>d$, satisfying
\begin{equation*}\bLg\bPhi + (\bm{c}+\lamstr)\bPhi \,\le\,  0 \quad \text{in\ } 
(K\times\cS_1)^c\,.
\end{equation*}
Let $B$ be a ball large enough so that $B\Supset K\cup B_1$.
Choose $\varrho$ large enough to that
$$\sup_n\, \sup_j\, \sup_{\Bar B}\, \bigl((\bm{v}_n)_j-\varrho\Phi_j\bigr) \,\le\, 0\,.$$
As earlier, applying the maximum principle we see that $\bm{v}_n\le \varrho\bPhi$ in 
$B_n\times\cS$. Passing to limit, as $n\to\infty$, we obtain $\bm{v}\le\varrho\bPhi$.
This completes the proof.
\end{proof}

We continue with the proof of \cref{T1.5}.

%%%%%%%%%%%%%%%%%%%%%%%%%%%%%%%%%%%%%%%%%%%%%%%%%%%%%%%%%%%%%%%%%%%%%%%%%%%%%%%%%%%%%
\begin{proof}[\bf Proof of \cref{T1.5}]
First we show that (a)$\,\Rightarrow\,$(b). Take $\bm{h}\in\sB^+_0(\Rd\times\cS)$.
We need to
show that $\lamstr(\bm{c}+\bm{h})<\lamstr(\bm{c})$. Suppose, on the contrary, that
$\lamstr(\bm{c}+\bm{h})=\lamstr(\bm{c})$. Let $\widehat\bPsi$ be a principal eigenfunction 
corresponding to the eigenvalue $\lamstr(\bm{c}+\bm{h})$, that is,
$$\bLg\widehat\bPsi + (\bm{c}+\bm{h}+\lamstr(c))\widehat\bPsi \,=\, 0
\quad \text{in\ } \Rd\times\cS\,.$$
Let $\Phi_k\df\frac{\widehat\Psi_k}{\Psi^*_k}$. Then from \cref{L2.2} we obtain
\begin{equation}\label{ET1.5A}
\widetilde\bLg^{\bm{\psi}^*}\bPhi\,=\, -\bm{h}\bPhi\,\le\, 0\quad \text{in\ } 
\Rd\times\cS\,.
\end{equation}
Since $\widetilde\bLg^{\bm{\psi}^*}$ is recurrent, it follows from \cref{P2.1} that
$\bPhi$ is a constant. This implies from \eqref{ET1.5A} that $\bm{h}\bPhi=0$, which
contradicts the fact that $\bm{h}\neq 0$. Thus we must have 
$\lamstr(\bm{c}+\bm{h})<\lamstr(\bm{c})$.

Next, we prove that (b)$\,\Rightarrow\,$(a). From \cref{T1.8} we know that monotonicity property
on the right is equivalent to the minimal growth at infinity of the principal 
eigenfunction. Therefore, we assume that the principal eigenpair $(\bPsi^*, \lamstr)$
has minimal growth property at infinity and show that (a) holds.
Suppose, on the contrary, that $\widetilde\bLg^{\bm\psi^*}$ is not recurrent.
In view of step 2 of the proof of \cref{T1.1}, 
we can find a non-zero bounded solution $\bm{u}$ of 
$$\widetilde\bLg^{\bm\psi^*}\bm{u} \,=\, 0\quad \text{in\ } (B\times\{i\})^c,
\quad \bm{u}=0\quad \text{in\ } \Bar{B}\times\{i\},$$
for some ball $B$ and some $i\in\cS$. We may also choose $B$ small enough so that 
$\bm{M}$ is
irreducible in $\Rd\setminus B$.
Scale $\bm{u}$
in such a fashion that $\sup_j\sup_{\Rd\,}u_j=1$. Since $\bm{u}$ is non-zero, this
supremum value is not attained in $\Rd\times\cS$. 
Define $\Phi_k\df\Psi^*_k\,(1-u_k)$.  Then $\bPhi$ is continuous and positive in 
$\Rd\times\cS$. Furthermore, $\bPhi\in\Sobl^{2, p}((B\times\{i\})^c)$ for $p>d$,
and applying \cref{L2.2} we also have
$$\bLg\bPhi \,=\, (1-\bm{u})\bLg\bPsi^* - \bPsi^*\,\widetilde\bLg^{\bm\psi^*}\bm{u}
\,=\, -(\bm{c}+\lamstr)\bPhi \quad \text{in\ } \bigl(B\times\{i\}\bigr)^c.$$
Since $\sup_j\sup_{\Rd\,}u_j=1$ there is no constant $\kappa>0$ satisfying
$(1-\bm{u})\bPsi^*=\bPhi\ge \kappa\bPsi^*$ in $\Rd\times\cS$. This contradicts
the minimal growth hypothesis of $\bPsi^*$. Hence $\widetilde\bLg^{\bm\psi^*}$ 
must be recurrent, and this implies (a).

Finally, simplicity of $\lamstr$ follows from \cref{T1.7}.
\end{proof}

We continue with the proof of \cref{T1.6}.

%%%%%%%%%%%%%%%%%%%%%%%%%%%%%%%%%%%%%%%%%%%%%%%%%%%%%%%%%%%%%%%%%%%%%%%%%%%%%%%%%%%%%
\begin{proof}[\bf Proof of \cref{T1.6}]
First we show that if (b) holds then there exists a principal eigenfunction
$\bPsi^*$ such that for some Lyapunov function $\bm\Lyap$ we have
\begin{equation}\label{ET1.6A}
\widetilde\bLg^{\bm\psi^*}\bm\Lyap \, \le\, \kappa_0\,\Ind_{K\times\cS} -\kappa_1 \bm\Lyap 
\quad \text{in\ } \Rd\times\cS\,,
\end{equation}
for some constants $\kappa_0, \kappa_1$ and a compact set $K$. Let 
$\bm{h}\in\sB^+_0(\Rd\times\cS)$ be such that $\lamstr(\bm{c}-\bm{h})>\lamstr(\bm{c})$,
and $(\bPsi_{\bm h}, \lamstr(\bm{c}-\bm{h}))$ be a principal eigenpair
with respect to  the potential $\bm{c}-\bm{h}$. 
Let $2\delta=\lamstr(\bm{c}-\bm{h})-\lamstr(\bm{c})$ and $\sB$ be such that 
$\sup_{\sB^c\times\cS} \bm{h}(x,i)\le \delta$.

Recall that 
$(\bPsi_n,\lambda_n)$ is the Dirichlet principal eigenpair in the ball $B_n$, that is,
\begin{equation*}
\bcA\bPsi_{n} \,=\, -\lambda_{n}\,\bPsi_n\,,
\quad \text{in\ } B_n\times\cS, \quad \text{and}\quad \bPsi_n=0\quad \text{on\ } 
\partial B_n\times\cS\,.
\end{equation*}
Define
\begin{equation*}
\kappa_n \,\df\, \sup\,\{\kappa>0\,\colon \bPsi_{\bm h}>\kappa\bPsi_n
\text{\ in\ } B_n\}\,.
\end{equation*}
Thus $\kappa_n\bPsi_n$ must touch $\bPsi_{\bm h}$ at some point from below.
We claim that any such point must lie in $\Bar\sB\times\cS$.
Note that for $x\in\sB^c$ we have 
\begin{equation*}
(\bLg(\bPsi_{\bm h}-\kappa_n\bPsi_n))_k +(c_k-\lambda_n)((\bPsi_{\bm{h}})_k
-\kappa_n(\bPsi_{n})_k) \,\le\, 0\,,
\end{equation*}
for all $k$, for all $n$ large so that $\lamstr(\bm{c}-\bm{h})>\lambda_n+\delta$. If 
$\kappa_n\bPsi_n$ 
touches
$\bPsi_{\bm h}$ in $B_n\setminus\sB$,
then by strong maximum principle some component of $\bPsi_{\bm h}-\kappa_n\bPsi_n$
must vanish in $B_n\setminus\sB$ which is not
possible since $\bPsi_{\bm h}-\kappa_n\bPsi_n> 0$ on $\partial{B}_n\times\cS$.
This proves the claim.
Therefore, using Harnack's inequality \cite{Sirakov}, we take limits as $n\to\infty$,
and see that $\kappa_n\bPsi_n$ converges
to the principal eigenfunction $\bPsi^*$ which touches $\bPsi_{\bm h}$ from below in
$\Bar{\sB}\times\cS$. In particular, we have $\bPsi^*\le \bPsi_{\bm h}$.

Now define $\Lyap_k(x) \df \frac{\bPsi_{\bm{h}}(x, k)}{\bPsi^*(x, k)}$. It is
then evident that $\bm\Lyap\ge 1$.
Then applying \cref{L2.2} we deduce that
\begin{equation*}
\widetilde\bLg^{\bm\psi^*} {\bm\Lyap} \,=\, (\bm{h} - \lamstr(\bm c - \bm h) + 
\lamstr)\bm\Lyap 
\le \kappa_0 \Ind_{\sB\times\cS} -\delta\bm\Lyap\, \quad \text{in\ }
\Rd\times\cS\,,
\end{equation*}
for some constant $\kappa_0$. This proves \eqref{ET1.6A}.

We continue by showing that (a)$\,\Rightarrow\,$(c). In fact, we show
that the Lyapunov condition \eqref{ET1.6A}
implies (c).  Suppose, on the contrary, that
for some $\bm{h}\in\sB^+_0(\Rd\times\cS)$ we have $\lamstr(\bm{c}-\bm{h})=
\lamstr(\bm{c})$.
Without any loss of generality, we may assume $\bm{h}$ is compactly 
supported.
Let $(\bm\xi_n, \beta_n)$ be the sequence of Dirichlet eigenpairs
corresponding to the potential $\bm{c}-\bm{h}$ in $B_n$. That is
\begin{equation}\label{ET1.6B}
\begin{split}
\bLg\bm\xi_n + (\bm{c}-\bm{h}+\beta_n)\bm\xi_n &\,=\, 0 \quad \text{in\ } B_n\times\cS,
\\
\bm\xi_n &\,>\, 0 \quad \text{in\ } B_n\times\cS,
\\
\bm\xi_n &\,=\, 0 \quad \text{on\ } \partial{B}_n\times\cS.
\end{split}
\end{equation}
Then we have $\beta_n\to\lamstr$.
Let $(\bm{u}_n)_k \df\frac{(\bm\xi_n)_k}{\Psi^*_k}$. Using \cref{L2.2} and 
\eqref{ET1.6B} we then
obtain that
\begin{equation}\label{ET1.6C}
\widetilde\bLg^{\bm\psi^*} \bm{u}_n\,=\, (\lamstr-\beta_n + \bm{h}) \bm{u}_n
\quad \text{in\ } B_n\times\cS\,.
\end{equation}
Recall the compact set $K$ from \eqref{ET1.6A}, and scale $\bm{u}_n$ to satisfy
$\max_{j}\max_K (\bm{u}_n)_j=1$. Let $\bm{u}$ be the limit of $\bm{u}_n$,
which exists due to Harnack inequality \cite{Sirakov} and \eqref{ET1.6C}. 
We claim that $\bm{u}\le \bm\Lyap^\alpha$ in $\Rd\times\cS$, for any
$\alpha\in (0,1)$, where 
$\bm\Lyap^\alpha=(\Lyap_1^\alpha, \Lyap_2^\alpha, \ldots, \Lyap^\alpha_N)$.
To prove the claim, first note that in $(K\times\cS)^c$ we have
\begin{equation}\label{ET1.6D}
\widetilde\bLg^{\bm\psi^*}\,\bm\Lyap^\alpha\,\le\, -\alpha \kappa_1 \bm\Lyap^\alpha.
\end{equation}
Indeed, for $\alpha\in(0, 1)$, we have from H\"{o}lder's inequality that
$$
\Lyap^\alpha_k \,=\, \Lyap^\alpha_k\, \Lyap^{\alpha(\alpha-1)}_i\, 
\Lyap^{\alpha(1-\alpha)}_i \,\le\, 
\alpha\, \Lyap_k\, \Lyap^{(\alpha-1)}_i + (1-\alpha) \Lyap^{\alpha}_i\quad
\text{for all\ } i, k\in\cS\,.
$$
Therefore, in $(K\times\cS)^c$, we get
\begin{align*}
\left(\widetilde\bLg^{\bm\psi^*}\,\bm\Lyap^\alpha\right)_i
&\,=\, \alpha \Lyap^{\alpha-1}_i (\widetilde\bLg^{\bm\psi^*}\,\bm\Lyap)_i
- \alpha \Lyap^{\alpha-1}_i \sum_{k\neq i} \tilde{m}_{ik} (\Lyap_k-\Lyap_i)
+ \alpha (\alpha-1) \Lyap_i^{\alpha-2}\,\grad\Lyap_i a_i\cdot\grad\Lyap_i
\\
&\mspace{100mu} + \sum_{k\neq i} \tilde{m}_{ik} (\Lyap^\alpha_k-\Lyap^\alpha_i)
\\
&\,\le\, -\alpha\kappa_1 \Lyap^\alpha_i + 
\sum_{k\neq i} \tilde{m}_{ik} (\Lyap^\alpha_k-\alpha\Lyap_k \Lyap_i^{\alpha-1})
- (1-\alpha) \sum_{k\neq i} \tilde{m}_{ik} \Lyap_i^\alpha
\\
&\,\le\, -\alpha\kappa_1 \Lyap^\alpha_i\,.
\end{align*}
Next, enlarge $K$  to contain support of $\bm{h}$. Now suppose that 
$\bm{u}_n>\bm\Lyap^\alpha$ 
at some point in $B_n\times\cS$. Since $\bm\Lyap\ge 1$ and 
$\max_j\max_K (\bm{u}_n)_j= 1$, there should be a point in 
$K^c\times\cS$ where $\bm{u}_n>\bm\Lyap^\alpha$. Choose $\kappa\in (0, 1)$ so that 
$\kappa \bm{u}_n\le \bm\Lyap^\alpha$ in $B_n\times\cS$ and 
$\kappa \bm{u}_n(x_0, i_0)=\bm\Lyap^\alpha(x_0, i_0)$ for some $x_0\in B_n\cap K^c$
and some $i_0\in\cS$. Using \eqref{ET1.6C} and \eqref{ET1.6D}, we obtain
\begin{equation*}
\widetilde\bLg^{\bm\psi^*} (\bm\Lyap^\alpha-\kappa \bm{u}_n)
\,\le\, -\alpha \kappa_1 \bm\Lyap^\alpha - (\lamstr - \beta_n + \bm{h})\kappa \bm{u}_n
\,\le\, \bigl(-\alpha\kappa_1  + |\lamstr-\beta_n|\bigr) \kappa \bm{u}_n \,<\, 0
\end{equation*}
for all large enough $n$.
Thus, writing $\bm\zeta= \bm\Lyap^\alpha-\kappa \bm{u}_n$, we see from above that
$$ \trace\bigl(a_{i_0}\grad^2 \zeta_{i_0}\bigr)
+ (b_{i_0} + 2a_{i_0}\grad\psi^*_{i_0})\cdot \grad \zeta_{i_0}
+  \widetilde{m}_{i_0i_0}(x)\zeta_{i_0}
\,\le\,\bigl(\widetilde\bLg^{\bm\psi^*}\bm\zeta\bigr)_{i_0}\,\le\, 0\quad \text{in\ } B_n\cap K^c.$$
Since $\zeta_{i_0}(x_0)=0$,  we must have 
$\zeta_{i_0}=0$ in $B_n\cap K^c$ by the strong maximum principle.
But this is not possible since $\zeta_{i_0}>0$ on
$\partial B_n$. Therefore, we must have $\bm{u}_n\le\bm\Lyap^\alpha$ in $B_n\times\cS$. 
Now letting $n\to\infty$, we establish the claim of $\bm{u}\le \bm\Lyap^\alpha$ in $\Rd\times\cS$, 
for any $\alpha\in (0,1)$.

Letting $\alpha\to 0$
we get $\bm{u}\le 1$. This of course, implies that $\bm{u}$ attains its maximum in
$K\times\cS$. Since 
$$\widetilde\bLg^{\bm\psi^*} \bm{u}\,=\, \bm{h}\bm{u}\,\ge\, 0\quad \text{in\ } \Rd\times\cS\,,$$
we must have $\bm{u}=(1, 1,\ldots, 1)$ by the strong maximum principle. But this
implies $\bm{h}\bm{u}=0$ which is not possible since $\bm{h}\neq 0$. Therefore,
the original hypothesis
that $\lamstr(\bm{c}-\bm{h})=\lamstr(\bm{c})$ cannot be correct. This 
establishes the inequality $\lamstr(\bm{c}-\bm{h})>\lamstr(\bm{c})$, giving us (c).

It is obvious that (c)$\,\Rightarrow\,$(b). Next suppose (b) holds. Then as have shown 
in \eqref{ET1.6A}, a Lyapunov function exists for $\widetilde\bLg^{\bm\psi^*}$. Applying
the preceding argument we see that $\lamstr(\bm{c}-\bm{h})>\lamstr(\bm{c})$ for every 
$\bm{h}\in\sB^+_0(\Rd\times\cS)$. 
Since $\RR\ni s\mapsto \lamstr(\bm{c}+s\bm{h})$ is decreasing and 
concave (\cref{L2.3}), we get that $\lamstr(\bm{c}+\bm{h})<\lamstr(\bm{c})$ for all 
$\bm{h}\in\sB^+_0(\Rd\times\cS)$. 
From \cref{T1.5}
we then see that $\widetilde\bLg^{\bm\psi^*}$ is recurrent and hence regular.
Combining this with \eqref{ET1.6A}
we get (a). This completes the proof.
\end{proof}
The remaining part of this section is devoted to the proof of \cref{T1.9}.
Let us first introduce the {\it twisted switching diffusion process}. Given
an eigenpair $(\Psi, \lambda)$ we recall the twisted operator from \eqref{E-tbLg},
which was defined as
\begin{equation}\label{E2.23}
\widetilde\bLg^{\bm\psi} \bm{f} (x) \,\df\, \bm{L}^{\bm\psi}\bm{f}(x)
+ \widetilde{\bm{M}}(x)\bm{f}(x)\,, \quad x\in\Rd\,.
\end{equation}
The corresponding {\it twisted switching diffusion} is
defined by
\begin{equation}\label{Sw-diff}
\begin{aligned}
\D \Tilde X_t &\,=\, b(\Tilde X_t, \Tilde S_t)\,\D t
+ 2 a(\Tilde X_t, \Tilde S_t) \grad\bm\psi (\Tilde X_t, \Tilde S_t) \D{t} +
\upsigma(\Tilde X_t, \Tilde S_t)\, \D \widetilde W_t\,,\\
\D \Tilde S_t &\,=\, \int_{\RR} \Tilde h(\Tilde X_t,\Tilde S_{t^-},z)
\widetilde\wp(\D t, \D z)
\end{aligned}
\end{equation}
for $t\ge 0$, where
\begin{itemize}
\item[(i)]
$\widetilde W$ is a $d$-dimensional standard Wiener process;
\item[(ii)]
$\widetilde\wp(\D t, \D z)$ is a Poisson random measure on $\RR_{+}\times\RR$
with intensity $\D t\times \D z$;
\item[(iii)]
$\widetilde\wp(\cdot,\cdot)$ and $\widetilde W(\cdot)$ are independent;
\item[(iv)]
The function $\Tilde h\colon\Rd\times\cS\times\RR \to \RR$ is defined by
\begin{equation*}
\Tilde h(x,i,z)\,\df\,\begin{cases}
j - i & \text{if}\,\, z\in \Vt_{ij}(x),\\
0 & \text{otherwise},
\end{cases}   
\end{equation*}
where for $i,j\in\cS$ and fixed $x$, $\Vt_{ij}(x)$ are left closed right open
disjoint intervals of $\RR$ having length
$\widetilde{m}_{ij}(x)=\frac{\Psi_j(x)}{\Psi_i(x)}m_{ij}(x)$ for $j\neq i$, and
\begin{equation*}
\widetilde{m}_{ii}(x) \,=\, -\sum_{j\neq i} \widetilde{m}_{ij}(x)\,.
\end{equation*}
\end{itemize}
In what follows, we use the notation $\widetilde\Prob$ and $\widetilde\Exp$
to denote the law of the twisted switching diffusion and the expectation with respect
to this probability measure, respectively.
It is clear that the extended generator of $(\tilde{X}, \tilde{S})$ is given by 
\eqref{E2.23}.
We write the generator simply as $\widetilde\bLg$ when the dependence on $\bm\psi$ is 
clear. Below we recall the probabilistic definition of 
recurrence and transience of a regime switching diffusion. For
a process $\{(X_t,S_t)\}_{t\ge0}$ on $\Rd\times\cS$,
denote by $\uptau(A)$ the first exit time
from the set $A\subset\RR^{d}\times\cS$, defined by
\begin{equation*}
\uptau(A) \,\df\, \inf\,\bigl\{t>0\,\colon (X_t,S_t)\not\in A\bigr\}\,.
\end{equation*}
The open ball of radius $r$ in $\RR^{d}$, centered at the origin,
is denoted by $B_{r}$, and we let $\uptau_{r}\df \uptau(B_{r})$,
$\uuptau_{r}\df \uptau(B^{c}_{r})$,  and
$\uuptau_{r}^{i}\df \uptau\bigl((\sB_{r}\times\{ i \})^{c}\bigr)$ for $i\in\cS$.

%%%%%%%%%%%%%%%%%%%%%%%%%%%%%%%%%%%%%%%%%%%%%%%%%%%%%%%%%%%%%%%%%%%%%%%%%%%%%%%%%%%%
\begin{definition}
The regime switching diffusion is said to be \emph{recurrent} if for any ball
$\sB\Subset\Rd$ and $j\in\cS$ we have
\begin{equation*}
\Prob_{x,i}(\uuptau(\sB\times\{j\})<\infty) \,=\, 1\,,
\end{equation*}
for all $(x,i)\in\Rd\times\cS$.
It is said to be \emph{transient} if for all $(x,i)\in\Rd\times\cS$ 
we have
\begin{equation*}
\Prob_{x,i}\left(\lim_{t\to\infty} |\Tilde{X}(t)|=\infty\right) \,=\, 1\,.
\end{equation*}
Given a set $A$ and $j\in\cS$, we define the hitting time to $A\times\{j\}$ by
$$\uuptau(A, j)\,\df\, \inf\,\bigl\{t>0\,\colon (X_t, S_t)\in A\times\{j\}\bigr\}\,.$$ 
\end{definition}
It is well known that a regime switching diffusion is either transient or recurrent 
\cite[Chapter~5]{book}, \cite[Chapter~3]{YZ10}
(the irreducibility of $\bm{M}$ is crucial for this statement).
Furthermore, a regime switching diffusion is transient if and only if for every $\bm 
g\in\cC_c(\Rd\times\cS)$ we have
$$\Exp_{x,i}\left[\int_0^\infty \bm g(X_t, S_t) \D{t}\right] \,<\, \infty$$
for all $(x,i)$ (cf. \cite[Proposition~3.19]{YZ10}). Furthermore, as discussed in 
\cref{R1.1}, the recurrence of $(X, S)$ is equivalent to the recurrence of its
extended generator $\bLg$ in the sense of \cref{D1.2}.

%%%%%%%%%%%%%%%%%%%%%%%%%%%%%%%%%%%%%%%%%%%%%%%%%%%%%%%%%%%%%%%%%%%%%%%%%%%%%%%%%%%%

The next lemma is the heart of the proof of \cref{T1.9}. It is basically a change of 
measure type result.
This should be compared with \cite[Lemma~2.3]{ABS19}.
In the case of scalar equation this change of measure 
is a straight-forward application of Girsanov's transformation.
But in the case of a regime switching diffusion
such a transformation is not known for the general model given above.
The Girsanov transformation for a simpler
setting can be found in \cite[Theorem~3.2]{Chan99}.

%%%%%%%%%%%%%%%%%%%%%%%%%%%%%%%%%%%%%%%%%%%%%%%%%%%%%%%%%%%%%%%%%%%%%%%%%%%%%%%%%%%%
\begin{lemma}\label{L2.4}
For any $\bm g\in\cC_{c}(\Rd\times\cS)$, and any eigenpair $(\bPsi,\lambda)$,
that is, a pair satisfying $\bm\cA\bPsi+\lambda\bPsi=0$, we have
\begin{equation}\label{EL2.4A}
\Exp_{x,k} \left[ \E^{\int_0^{T} (\bm c(X_s,S_s)+\lambda) \D{s}}\,
\bm g(X_{T}, S_{T}) \bPsi(X_{T}, S_{T})\,\right]
\,=\, \bPsi(x,k)\,\widetilde\Exp_{x,k}^{\bm\psi}
\left[\bm g(\Tilde X_{T}, \Tilde S_{T})\Ind_{\{T < \uptau_\infty\}}\,\right]\,,
\end{equation}
for all $(x,k)\in\Rd\times\cS$,
where $\widetilde\Exp_{x,k}^{\bm\psi}$ is the expectation operator with respect to
the law of the twisted process \cref{Sw-diff}, and
$\uptau_\infty \df \lim_{n\to\infty}\uptau_{n}$ is the explosion time.
\end{lemma}

\begin{proof}
Let $\bm g\in\cC_{c}^{2}(\Rd\times\cS)$ with $\supp{\bm g}\subset\sB\times\cS$
for some ball $\sB\subset\Rd$.
Select $n_{0}\in\NN$ sufficiently large such that
$\sB\times\cS\subset B_n\times\cS$ for all $n\ge n_{0}$.
Now fix any $n\ge n_{0}$.
By \cite[Theorem~1.1, p.~573]{LSU-86}, there exist
$\bm\phi\in\Sob^{1,1,2}\bigl((0,T)\times B_n\times\cS\bigr)$ satisfying 
\begin{equation}\label{EL2.4B}
\begin{aligned}
\frac{\partial{\phi}_{k}}{\partial t}(t, x) - (\widetilde\bLg\bm\phi)_k(x,t) &\,=\, 0\,,
\\
\phi_{k}(x,0) &\,=\, \bm g (x,k)\,,
\\
\phi_{k}(x,t) &\,=\, 0 \quad\text{on\ } \partial{B_n}\times\cS\times [0,T]\,,
\end{aligned}
\end{equation}
for all $k\in\cS$,
where 
\begin{equation}\label{EL2.4C}
\begin{aligned}
(\widetilde\bLg\bm\phi)_k\,=\,(\widetilde\bLg^{\bm\psi}\bm\phi)_k &\,=\, 
\trace\left(a_k\nabla^{2}\phi_{k}\right)
+ \bigl\langle b_k, \nabla \phi_{k}\bigr\rangle
+ 2\langle \grad\psi_k,a_{k}\grad\phi_k\rangle\\
&\mspace{150mu}+ \sum_{j\ne k} m_{kj}\frac{\Psi_{j}}{\Psi_{k}}\bigl(\phi_j - \phi_k\bigr)\,,
\quad \forall\, k\in\cS\,.
\end{aligned}
\end{equation}
By the Gagliardo--Nirenberg--Sobolev inequality, we have
$\bm\phi\in L^{2^{*}}\bigl((0,T)\times B_n\times\cS\bigr)$.
Then considering \cref{EL2.4B} as an equation in $\phi_{k}$, it is clear from 
\cite[Theorem~9.2.5]{WYW-06} and \cite[Theorem~3.4, p.~89]{ALAIN-82}
that $\phi_{k}\in\Sob^{1,2,2^{*}}\bigl((0,T)\times B_n\bigr)$ for all $k\in\cS$.
Repeating the above argument it is easy to see that
$\bm\phi\in\Sob^{1,2,p}\bigl((0,T)\times B_n\times\cS\bigr)
\cap\cC\bigl([0,T]\times\Bar{\sB}_{n}\times\cS\bigr)$, $p\ge d$.
Now using \cref{EL2.4B}, and applying the It\^{o}--Krylov formula \cite[p.~122]{Krylov} on
$\bm\phi(\Tilde X_t,\Tilde S_t, T-t)$, it follows that
\begin{equation}\label{EL2.4D}
\begin{aligned}
\bm\phi(x,k,T) &\,=\, \widetilde\Exp_{x,k}^{\bm\psi}
\left[\bm\phi\bigl(\Tilde X_{T\wedge\uptau_n},
\Tilde S_{T\wedge\uptau_n}, T - T\wedge\uptau_n\bigr)\right] \\
&\,=\, \widetilde\Exp_{x,k}^{\bm\psi} \left[\bm g(\Tilde X_{T},\Tilde S_{T})
\Ind_{\{T < \uptau_n\}}\right]\,.
\end{aligned}
\end{equation}
Rewriting \cref{EL2.4C}, we obtain
\begin{equation}\label{EL2.4E}
\Psi_{k}\widetilde\Lg_{k}\bm\phi \,=\,
\Psi_{k}\bigl[\trace\left(a_k\nabla^{2}\phi_{k}\right)
+ \bigl\langle b_k, \nabla \phi_{k}\bigr\rangle
+ 2\langle \grad\psi_k,a_{k}\grad\phi_k\rangle \bigr]
+ \sum_{j\ne k} m_{kj}\bigl(\Psi_{j}\phi_j - \Psi_{j}\phi_k\bigr)
\end{equation}
for all $k\in\cS$.
Let $\Hat{\phi}_{k} = \phi_{k}\Psi_{k}$. Then, using \cref{L2.1}, we get
\begin{align*}
(\bLg\Hat{\bm\phi})_k \,=\, (-\lambda - c_{k})\phi_{k}\Psi_{k}
+ \Psi_{k}(\widetilde\bLg\bm\phi)_k\,. 
\end{align*}
Therefore, from \cref{EL2.4B}, we deduce that
\begin{equation}\label{EL2.4F}
\begin{aligned}
\frac{\partial{\Hat{\phi}}_{k}}{\partial t}(x,t)
& \,=\,\Psi_{k}(x)(\widetilde\bLg\bm{\phi})_k(x,t) \\
& \,=\, (\bLg\Hat{\bm\phi})_k(x,t)
+ ( c_{k} + \lambda )\Hat{\phi}_{k}(x,t)\,, \quad \forall\, k\in\cS\,.
\end{aligned}
\end{equation}
Thus, since we have $\Hat{\bm\phi}(x,k,0) = \bm g(x,k)\bPsi(x,k)$,
and $\Hat{\bm\phi}(x,k,t) = 0$ on $\partial{B_n}\times\cS\times [0,T]$
for all $k\in\cS$, it follows from \eqref{EL2.4F} by
an application of the It\^{o}--Krylov formula, that
\begin{align}\label{EL2.4G}
\bm\phi(x,k,T)\bPsi(x,k) &\,=\, \Exp_{x,k}\left[\E^{\int_{0}^{T\wedge\uptau_{n}}
\left(\bm c(X_{t},S_{t}) + \lambda \right)\,\D{t}}
\bm\phi(X_{T\wedge\uptau_{n}},S_{T\wedge\uptau_{n}}, T - T\wedge\uptau_{n})
\bPsi(X_{T\wedge\uptau_{n}},S_{T\wedge\uptau_{n}})\right]\nonumber\\
&\,=\, \Exp_{x,k}\left[\E^{\int_{0}^{T}\left(\bm c(X_{t},S_{t}) + \lambda \right)\,\D{t}}
\bm g(X_{T},S_{T})\bPsi(X_{T},S_{T})\Ind_{\{T < \uptau_{n}\}}\right].
\end{align}
Now combining \cref{EL2.4D,EL2.4G,EL2.4E}, we obtain 
\begin{equation}\label{EL2.4H}
\bPsi(x,k)\,\widetilde\Exp_{x,k}^{\bm\psi}
\left[\bm g(\Tilde X_{T},\Tilde S_{T}) \Ind_{\{T < \uptau_n\}}\right]
\,=\, \Exp_{x,k}\left[\E^{\int_{0}^{T}\left(\bm c(X_{t},S_{t}) + \lambda \right)\,\D{t}}
\bm g(X_{T},S_{T})\bPsi(X_{T},S_{T})\Ind_{\{T < \uptau_{n}\}}\right].
\end{equation}
Applying the monotone convergence theorem to take limits as $n\to\infty$ in \cref{EL2.4H},
we obtain \cref{EL2.4A} for $\bm g\in\cC_{c}^{2}(\Rd\times\cS)$.
A standard approximation argument shows that \cref{EL2.4A} also holds
for $\bm g\in\cC_{c}(\Rd\times\cS)$.
This completes the proof. 
\end{proof}

%%%%%%%%%%%%%%%%%%%%%%%%%%%%%%%%%%%%%%%%%%%%%%%%%%%%%%%%%%%%%%%%%%%%%%%%%%%%%%%%%%%%

Finally, we prove \cref{T1.9}.
For a ball $\sB$, centered at $0$, we denote by
$\uuptau$ the first hitting time to $\sB\times\cS$.

%%%%%%%%%%%%%%%%%%%%%%%%%%%%%%%%%%%%%%%%%%%%%%%%%%%%%%%%%%%%%%%%%%%%%%%%%%%%%%%%%%%%
\begin{proof}[\bf Proof of \cref{T1.9}]
First we show that (a)$\,\Rightarrow\,$(b).
From \cref{T1.5} we know that $(\Tilde{X}, \Tilde{S})$ is recurrent.
Thus for some compactly support $\bm{g}\in C^+_c(\Rd\times\cS)$ we have
\begin{equation*}
\widetilde\Exp_{x, k}
\left[\int_0^\infty \bm{\xi}(\Tilde{X}_t,\Tilde{S}_t)\, \D{t}\right] \,=\, \infty
\end{equation*}
for some $(x, k)$, where 
\begin{equation*}
\xi_k(x) \,\df\, \frac{g_k(x)}{\Psi^*_k(x)}\,, \quad k\in\cS\,.
\end{equation*} 
With no loss of generality we assume that $(x, k)=(0, 1)$. For $\alpha>0$ we define 
$\bm{F}_\alpha(x,i) \df \bm{c}(x,i)+\lamstr-\alpha$ and
\begin{equation*}
\Gamma_\alpha \,\df\,
\Exp_{0,1}\left[\int_0^\infty \E^{\int_0^t \bm{F}_\alpha(X_s, S_s)\, \D{s}} 
g(X_t, S_t)\right]\,.
\end{equation*}
Then as shown in \cite[Lemma~2.7]{ABS19},
we have $\Gamma_\alpha<\infty$ for all $\alpha>0$, and 
using \cref{L2.4}
we get $\Gamma_\alpha\to \infty$ as $\alpha\to 0$. Let 
$\bPhi^\alpha_n\in\cC_0(B_n\times\cS)\cap\Sobl^{2,p}(B_n\times\cS)$ 
be the unique solution to
\begin{equation}\label{ET1.9A}
\Lg_k\bPhi^\alpha_n + F_{\alpha, k} \Phi^\alpha_{n, k} \,=\,  - \Gamma^{-1}_\alpha g_k 
\quad \text{in\ } B_n,
\end{equation}
for all $k$. The existence follows from \cref{LA.3}. Again following the arguments in 
\cite[Lemma~2.7]{ABS19}  and the Harnack inequality in \cite[Theorem~2]{Sirakov},
we can show that the family $\{\bPhi^\alpha_n\}_{n\ge n_0}$ is locally
uniformly bounded in $\Sob^{2,p}$-norm and therefore, we can extract a subsequence 
converging to a positive $\bPhi^\alpha\in\Sobl^{2,p}(\Rd)$, $p>d$, satisfying
\begin{equation}\label{ET1.9B}
\Lg_k\bPhi^\alpha + F_{\alpha, k} \Phi^\alpha_{ k} \,=\,  - \Gamma^{-1}_\alpha g_k \quad 
\text{in\ } \Rd\,,
\end{equation}
for all $k$. In addition, we also have $\bPhi^{\alpha}(0, 1)=1$ for all $\alpha>0$. 
Let $\sB$ be a ball satisfying $\supp(\bm{g})\Subset \sB\times\cS$.
Then applying the It\^{o}--Krylov formula to \eqref{ET1.9A} we obtain
\begin{equation*}
\bPhi^\alpha_n(x, k) \,=\,
\Exp_{x,k}\left[\E^{\int_0^{T\wedge\uuptau}\bm{F}_\alpha(X_t, S_t)\, 
\D{t}} \bPhi^\alpha_n(X_{T\wedge\uuptau}, S_{T\wedge\uuptau})
\Ind_{\{T\wedge\uuptau<\uptau_n\}}\right], \quad \text{for\ } x\in B_n\setminus\Bar\sB\,.
\end{equation*}
As shown in \cite[Lemma~2.7]{ABS19}, we can let $T\to\infty$ first and then $n\to\infty$ to 
arrive at 
\begin{equation}\label{ET1.9C}
\bPhi^\alpha(x, k) \,=\,
\Exp_{x, k}\left[\E^{\int_0^{\uuptau}\bm{F}_\alpha(X_t, S_t)\, \D{t}} 
\bPhi^\alpha(X_{\uuptau}, S_{\uuptau})
\Ind_{\{\uuptau<\infty\}}\right], \quad \text{for\ } x\in \sB^c.
\end{equation}
Since $\bPhi^\alpha(0, 1)=1$, using Harnack's inequality  and the
Sobolev estimate we can extract a subsequence of $\{\bPhi^\alpha\}_{\alpha\in (0, 1)}$.
converging  to $\bPsi^*$ in $\Sobl^{2,p}(\Rd)$ as $\alpha\to 0$. It is then evident from 
\eqref{ET1.9B} that 
\begin{equation*}
(\bLg\bPsi^*)_k + (c_k(x)+\lamstr) \Psi^*_{ k} \,=\,  0 \quad \text{in\ } \Rd\,,\quad 
\forall\, k\in\cS\,,
\end{equation*}
and passing the limit in \eqref{ET1.9C} with the help of monotone convergence theorem we 
obtain 
\begin{equation}\label{ET1.9D}
\bPsi^*(x, k) \,=\,
\Exp_{x, k}\left[\E^{\int_0^{\uuptau}(\bm{c}(X_t, S_t)+\lamstr)\, \D{t}} 
\bPsi^*(X_{\uuptau}, S_{\uuptau})
\Ind_{\{\uuptau<\infty\}}\right], \quad \text{for\ } x\in \sB^c.
\end{equation}
Thus we get \eqref{ET2.7A}. This gives us (b).

Next we show that (b)$\,\Rightarrow\,$(a). Let $\bm{h}\in\sB^+_0(\Rd\times\cS)$, and
suppose, on the contrary, that $\lamstr(\bm{c}+\bm{h})=\lamstr(\bm{c})=\lamstr$.
Let $\widehat\bPsi> 0$ be an principal eigenfunction with potential $\bm{c}+\bm{h}$.
Then 
\begin{equation}\label{ET1.9E}
(\bLg\widehat\bPsi)_k + (c_k + \lamstr)\widehat\Psi_k \,\le\,
(\bLg\widehat\bPsi)_k + (c_k + h_k +\lamstr)\widehat\Psi_k\,=\,  0, \quad \text{in\ } 
\Rd\,, \quad \forall\, k\in\cS\,.
\end{equation}
Let $\sB$ be the ball given in (b) and $\uuptau$ be the first hitting time to $\sB\times\cS$. 
Then using It\^{o}--Krylov formula and Fatou's lemma to \eqref{ET1.9E} we obtain that 
\begin{equation}\label{ET1.9F}
\widehat\bPsi(x, k)\,\ge\, \Exp_{x,k}
\left[\E^{\int_0^{\uuptau}(\bm{c}(X_t,S_t)+\lamstr)\, 
\D{t}} \widehat\bPsi(X_{\uuptau}, S_{\uuptau})
\Ind_{\{\uuptau<\infty\}}\right], \quad \text{for\ } x\in \sB^c,\, k\in\cS\,.
\end{equation}
Define
\begin{equation*}
\kappa \,\df\, \min_{k}\, \min_{\Bar\sB}\frac{\widehat\Psi_k(x)}{\Psi^*_k}.
\end{equation*} 
Then using \eqref{ET1.9D} and \eqref{ET1.9F} we have $\widehat\bPsi\ge \kappa \bPsi^*$
and $\min_{k}\min_{\Bar\sB}(\widehat\Psi_k-\kappa\Psi^*)=0$. 
Since $\bLg(\widehat\bPsi-\kappa\bPsi^*)\le 0$, it follows from the strong maximum principle 
that $\widehat\bPsi=\kappa\bPsi^*$. From \eqref{ET1.9E} this also
gives us $\bm{h}\widehat\bPsi=0$, which contradicts the fact that $\bm{h}\neq 0$. 
Hence $\lamstr(\bm{c}+\bm{h})<\lamstr(\bm{c})$, establishing (a).
\end{proof}

%%%%%%%%%%%%%%%%%%%%%%%%%%%%%%%%%%%%%%%%%%%%%%%%%%%%%%%%%%%%%%%%%%%%%%%%%%%%%%%%%%%%
\appendix
\section{The Dirichlet eigenvalue problem in bounded domains}\label{A-Eigen}
In this section we consider the principal eigenvalue problem in a
smooth bounded domain $D\subset\Rd$.
Some of the results obtained below can also
be found in \cite{BS04} (see Theorems~13.1 and 13.2 there).
Without any loss of generality we may assume that $0\in D$.
For this section, the only hypotheses we require
are summarized in the following assumption.

%%%%%%%%%%%%%%%%%%%%%%%%%%%%%%%%%%%%%%%%%%%%%%%%%%%%%%%%%%%%%%%%%%%%%%%%%%%%%%%%%%%%
\begin{assumption}\label{AA.1}
The following hold.
\begin{itemize}
\item[(i)]
$\bm a\in\bigl(\cC(\overline{D}\times\cS)\bigr)^{d\times d}$,
and, for some $\Lambda > 0$, we have
\begin{equation*}
\Lambda^{-1} \Id \,\le\, a_{k}(x) \,\le\, \Lambda \Id
\quad\forall\,x\in\Bar D\,,\ \forall\,k\in\cS\,.
\end{equation*}
\item[(ii)]
$\bm b\colon D\times\cS\to\Rd$, $m_{ij}\colon D\to\RR$, $i,j\in\cS$,
and $\bm c\colon D\times\cS\to\RR$ are bounded, Borel measurable functions.
\item[(iii)]
$\bm{M}$ is irreducible in $D$, that is,
\cref{irred} holds.
\end{itemize}
\end{assumption}

For $\lambda\in\RR$, consider the set
\begin{equation*}
\Uppsi_{D}^+(\lambda) \,\df\, \bigl\{\bPhi\in \Sobl^{2,d}(D\times\cS)\cap C(\bar{D}\times\cS)
\,\colon\, \bPhi> 0 \text{\ in\ } D\times\cS\,,
\ (\bcA\bPhi)_{k}(x) + \lambda\bPhi_{k}(x)\,\le \,0
\text{\ in\ }  D \ \ \forall \,  k\in\cS \bigr\}\,.
\end{equation*}
We define the generalized Dirichlet principal eigenvalue $\lambda_D$ of $\bm\cA$
in the domain $D$ by
\begin{equation}\label{EA.1}
\lambda_D \,\df\, \sup\,\bigl\{\lambda\in\RR\,\colon
\Uppsi_{D}^+(\lambda)\ne \varnothing\bigr\}\,.
\end{equation}

%%%%%%%%%%%%%%%%%%%%%%%%%%%%%%%%%%%%%%%%%%%%%%%%%%%%%%%%%%%%%%%%%%%%%%%%%%%%%%%%%%%%
\begin{theorem}\label{TA.1}
There exists a unique pair
$(\bm\varphi, \rho)\in \cC_{0}(\overline{D}\times\cS)
\cap\Sobl^{2,p}(D\times\cS)\times\RR$, $p>d$,
satisfying
\begin{equation}\label{ETA.1A}
\begin{split}
\bcA \bm\varphi_{D} &\,=\, -\rho\, \bm\varphi_{D} \quad \mbox{in\ }
D\times\cS\,,\\
\bm\varphi_{D} &\,=\, 0\quad \text{on\ } \partial{D}\times\cS\,,\\
\bm\varphi_{D} &\,>\, 0 \quad \text{in\ } D\times\cS\,.
\end{split}
\end{equation}
In addition $\rho=\lambda_D$.
\end{theorem}

\begin{proof}
Since $\bm c$ is bounded, using $\bigl(\norm{\bm c}_{L^\infty(D\times\cS)} - \bm c\bigr)$
as the coefficient of the zeroth order term,
the existence of unique solution
$(\bm\varphi_{D},\rho) \in
\cC_{0}(\overline{D}\times\cS)\cap\Sobl^{2,p}(D\times\cS)\times\RR$, $p>d$, to \cref{ETA.1A}
follows from \cite[Remark~1.3 and Corollary 2.1]{Sweer}.
Uniqueness of $\bm\varphi_{D}$ is of course only up to a multiplicative constant.

We claim that if $\bm w\in \cC(\overline{D}\times\cS)\cap\Sobl^{2,p}(D\times\cS)$,
with $\bm w > 0$ in $D$, satisfies
\begin{equation*}
\bcA \bm w \,\le\,  -\rho\, \bm{w} \text{\ in\ } D\times\cS\,,
\end{equation*}
then $\bm w = t\bm \varphi_{D}$ for some constant $t > 0$.
This clearly implies that $\rho = \lambda_D$.

In order to prove the claim, we define $\bm u_{t} \,\df \, t\bm\varphi_{D} - \bm w$.
Let $K \subset D$ be a compact set such that $|D\setminus K| < \varepsilon$,
for some small number $\varepsilon > 0$.
Then for some suitable choice of $t > 0$ (small enough),
we have $\bm u_{t} \le 0$ in $K\times\cS$.
Also, for all $k\in\cS$ we have
\begin{equation*}
(\bLg \bm u_{t})_k(x) - (c_{k}(x)+\rho)^{-} (\bm{u}_{t})_{k}(x) 
\,\ge\, - (c_{k}(x)+\rho)^{+} (\bm{u}_{t})_{k}(x)\,.
\end{equation*}
Now choosing $\varepsilon$ sufficiently small, and
applying \cite[Theorem~1]{Sirakov} on the domain $D\setminus  K$
we see that
$$\sup_{D\setminus K}\,\max_{k\in\cS}\, (\bm{u}_{t})_{k}
\,\le\, \theta_0\,\sup_{D\setminus K}\,\max_{k\in\cS}\,(\bm{u}_{t})_{k}^{+}$$
for some $\theta_0\in(0,1)$.
This is possible only if $\bm u_t \le 0$ in $D\setminus K$.
Thus, $\bm u_t\le 0$ in $D\times\cS$, and it satisfies
\begin{equation*}
\trace (a_{k}\grad^2 (\bm{u}_{t})_{k})(x) + b_{k}(x)\cdot \grad (\bm{u}_{t})_{k}(x)
- (c_{k}(x) + m_{k,k}(x) + \rho)^{-} (\bm{u}_{t})_{k}(x) \,\ge\,  0\,.
\end{equation*}
Therefore, we must either have $(\bm{u}_{t})_{k}=0$ or $(\bm{u}_{t})_{k}<0$ in $D$
by the strong maximum principle \cite[Theorem~9.6]{GilTru}.
If $(\bm{u}_{t})_{k} < 0$ for some $k\in\cS$, the irreducibility condition in \cref{AA.1}\,(iii)
implies that $(\bm{u}_{t})_{j} < 0$ for all $j\in\cS$.
Thus we either have $\bm u_{t}=0$ or $\bm u_{t} < 0$ in $D\times\cS$.
Suppose that $\bm u_{t} < 0$ in $D\times\cS$. 
Define
\begin{equation*}
\mathfrak{t} \,\df\, \sup\,\{t>0 \,\colon\, \bm u_t < 0 \quad \text{in\ } D\times\cS\}\,.
\end{equation*}
By the above argument, $\mathfrak{t}>0$, and by the strong maximum principle,
we must have either $\bm u_\mathfrak{t}=0$
or $\bm u_\mathfrak{t}< 0$.
If $\bm u_\mathfrak{t} < 0$, then for some $\delta>0$ we have
$\bm u_{\mathfrak{t}+\delta} < 0$ in $K\times\cS$,
and repeating the argument above, we obtain $\bm u_{\mathfrak{t} + \delta} < 0$
in $D\times\cS$.
This contradicts the definition of $\mathfrak{t}$. So the only possibility is
$\bm u_\mathfrak{t} = 0$. This indeed implies that $\rho = \lambda_{D}$
and completes the proof.
\end{proof}

\Cref{TA.2,TA.3} which follow,
concern the strict monotonicity of the principal eigenvalue with respect
to the potential and the domain.
We denote the eigenvalue as $\lambda_D(\bm c)$ when we want to explicitly indicate
its dependence on the potential $\bm c$.

%%%%%%%%%%%%%%%%%%%%%%%%%%%%%%%%%%%%%%%%%%%%%%%%%%%%%%%%%%%%%%%%%%%%%%%%%%%%%%%%%%%%
\begin{theorem}\label{TA.2}
If two potentials satisfy $\bm c\lneq \bm c'$,
then $\lambda_D(\bm c) > \lambda_D(\bm c')$.
\end{theorem}

\begin{proof}
Let $\bm\varphi_{\bm c}$ and $\bm\varphi_{\bm c'}$ denote the principal eigenfunctions
corresponding to $\bm c$ and $\bm c'$, respectively.
It is clear from \cref{EA.1} that $\lambda_D(\bm c) \ge \lambda_D(\bm c')$.
Suppose that $\lambda_D(\bm c)=\lambda_D(\bm c')$. Then, we obtain
\begin{equation*}
\bcA\, \bm\varphi_{\bm c'}(x) \,\le\,  -\lambda_{D}(\bm c) \bm\varphi_{\bm c'}(x)\,,
\quad\text{in\ } D\times\cS\,.
\end{equation*}
Now, it follows from the proof of \cref{TA.1}  that
$\bm\varphi_{\bm c'} = t\bm\varphi_{\bm c}$ for some positive constant $t$.
But this contradicts the fact that $\bm c\lneq \bm c'$.
Therefore, we have $\lambda_D(\bm c) > \lambda_D(\bm c')$.
\end{proof}

%%%%%%%%%%%%%%%%%%%%%%%%%%%%%%%%%%%%%%%%%%%%%%%%%%%%%%%%%%%%%%%%%%%%%%%%%%%%%%%%%%%%
\begin{theorem}\label{TA.3}
If $D_{1}\subsetneq D_{2}$, then $\lambda_{D_1} > \lambda_{D_2}$.
\end{theorem}
\begin{proof}
Let $\bm \varphi_1$ and $\bm\varphi_2$ denote the principal eigenfunctions corresponding
to $\lambda_{D_1}$ and $\lambda_{D_2}$, respectively.
From the definition in \cref{EA.1}, it follows that $\lambda_{D_1} \ge \lambda_{D_2}$.
If $\lambda_{D_1} = \lambda_{D_2}$, then
\begin{equation*}
\bcA\, \bm\varphi_2 \,\le\,  -\lambda_{D_1} \bm\varphi_2
\quad\text{in\ } D_2\times\cS\,.
\end{equation*}
As in the proof of \cref{TA.1}, this implies $\bm\varphi_2 = t \bm\varphi_1$
on $D_1\times\cS$
for some $t > 0$. This contradicts the fact that $\bm\varphi_2 > 0$ in $D_{2}\times\cS$,
because $D_{1}\subsetneq D_{2}$ and $\bm\varphi_1 = 0$ on $\partial{D}_{1}\times\cS$. 
Thus, we must have $\lambda_{D_1} > \lambda_{D_2}$.
\end{proof}

Next, we address the continuity properties of the principal eigenvalue with respect
to the domain $D$.
We say that a domain $D$ has the \emph{exterior sphere property} of radius
$r>0$, if every point of $\partial D$ can be touched from outside of $D$ with a ball
of radius $r$.
 We need the following
boundary estimate. For a proof, see \cite[Lemma 6.1]{AB-19}.

%%%%%%%%%%%%%%%%%%%%%%%%%%%%%%%%%%%%%%%%%%%%%%%%%%%%%%%%%%%%%%%%%%%%%%%%%%%%%%%%%%%%
\begin{lemma}\label{LA.1}
Suppose that $\norm{\bm w}_{L^\infty(D\times\cS)} \le 1$, and it satisfies
\begin{equation*}
\trace (a_{k}\grad^2 w_{k}) + \delta\abs{\grad w_{k}} \,\ge\, L
\text{\ \ in\ } D\quad \forall\, k\in\cS\,, \quad\text{and\ \ } \bm w = 0
\text{\ \ on\ } \partial{D}\times\cS\,,
\end{equation*}
where $D$ has the exterior sphere property of radius $r>0$.
Then for $s\in(0,1)$, there exist constants $M$, and $\varepsilon$,
depending only on
$\delta$, $L$, $r$, and $s$, such that
\begin{equation*}
\max_{k\in\cS}\, w^+_{k}(x)\,\le\, M \dist(x,\partial D)^s, \quad \text{for all $x$ such 
that\ }
\dist(x, \partial D)<\varepsilon\,.
\end{equation*}
\end{lemma}

%%%%%%%%%%%%%%%%%%%%%%%%%%%%%%%%%%%%%%%%%%%%%%%%%%%%%%%%%%%%%%%%%%%%%%%%%%%%%%%%%%%%
\begin{theorem}\label{TA.4}
Let $\{D_n\}_{n\in\NN}$ be a decreasing sequence of smooth domains whose intersection
is denoted as $D$, and which have the exterior sphere property of radius $r$
uniformly in $n\in\NN$. Then $\lambda_{D_n}\to \lambda_D$, as $n\to\infty$.
\end{theorem}

\begin{proof}
From \cref{TA.3} it is clear that $\lambda_{D_n}$ is a increasing sequence
which is bounded above by $\lambda_D$.
Thus, $\lambda_{D_n}$ converges to some number $\Tilde\lambda\le \lambda_D$.
We normalize the eigenfunctions so that
$\norm{\varphi_{D_n}}_{L^\infty(D_{n}\times\cS)} = 1$.
Now, using \cref{LA.1} and the standard interior estimate,
it can be easily seen that the family
$\{\bm\varphi_{D_n}\}$ is equicontinuous and each limit point
$\bm\phi\in \cC(\overline{D}\times\cS)\cap\Sobl^{2,p}(D\times\cS)$ is a
nonnegative solution to
\begin{equation*}
\bcA\, \bm\phi \,=\, -\Tilde{\lambda}\, \bm\phi  \text{\ \ in\ } D\times\cS\,.
\end{equation*}
By the strong maximum principle, we must have $\bm\phi > 0$ in $D\times\cS$.
Thus, the equality $\Tilde\lambda=\lambda_D$ follows from the proof of \cref{TA.1}.
\end{proof}

The next result shows that $\lambda_D$ is convex with respect to the potential $c$.

%%%%%%%%%%%%%%%%%%%%%%%%%%%%%%%%%%%%%%%%%%%%%%%%%%%%%%%%%%%%%%%%%%%%%%%%%%%%%%%%%%%%
\begin{lemma}\label{LA.2}
It holds that
\begin{equation}\label{ELA.2A}
\lambda_D(\theta \bm c_{1} + (1-\theta) \bm c_2)
\,\ge\, \theta \lambda_D(\bm c_{1}) + (1-\theta) \lambda_D(\bm c_{2})
\quad\forall\,\theta\in [0,1]\,.
\end{equation}
\end{lemma}

\begin{proof}
Let $\bm \varphi_i$ denote the principal eigenfunction with respect
to the potential $\bm c_i$, $i=1,2$. Define
$f_k \,\df\, \varphi_{1, k}^\theta \varphi^{(1-\theta)}_{2, k}$.
Then, by Young's inequality we have
\begin{equation*}
\sum_{j\neq k}m_{kj}\frac{f_{j}}{f_{k}}
\,=\, \sum_{j\neq k}m_{kj}\frac{\varphi_{1, j}^\theta \varphi^{(1-\theta)}_{2, j}}
{\varphi_{1, k}^\theta \varphi^{(1-\theta)}_{2, k}}
\,\le\, \sum_{j\neq k}m_{kj}\left(\theta\frac{\varphi_{1, j}}{\varphi_{1, k}}
+ (1-\theta) \frac{\varphi_{2, j}}{\varphi_{2, k}}\right).
\end{equation*}
Also, it is straightforward to show that
\begin{equation*}
\frac{1}{f_{k}}\trace(a_{k}\grad^2 f_{k})
\,\le\, \frac{\theta}{\varphi_{1,k}}\trace(a_{k}\grad^2 \varphi_{1,k})
+ \frac{(1 - \theta)}{\varphi_{2,k}}\trace(a_{k}\grad^2 \varphi_{2,k})
\quad \forall\,k\in\cS\,.
\end{equation*}
Thus,  we obtain
\begin{equation*}
\frac{1}{f_k} \left(\Lg_k\bm f + \left(\theta c_{1,k}
+ (1-\theta) c_{2,k}\right)f_{k} \right)
\,\le\,  \frac{\theta}{\varphi_{1,k}} \left(\Lg_k\bm\varphi_{1}
+  c_{1,k}\,\varphi_{1,k}\right)
+ \frac{(1-\theta)}{\varphi_{2,k}} \left(\Lg_k\bm\varphi_{2}
+  c_{2,k}\,\varphi_{2,k}\right).
\end{equation*}
for all $k\in\cS$.
Simplifying the above inequality, we obtain
\begin{equation*}
\Lg_k\bm f + \left(\theta \bm c_{1,k} + (1-\theta) \bm c_{2,k}\right)f_{k}
\,\le \, -\bigl(\theta \lambda_D(\bm c_{1}) + (1-\theta) \lambda_D(\bm c_{2})\bigr)
f_{k}\quad\forall\,\, k\in\cS\,.
\end{equation*}
In view of \eqref{EA.1}, this implies \cref{ELA.2A}.
\end{proof}

We conclude the Appendix with the following result.

%%%%%%%%%%%%%%%%%%%%%%%%%%%%%%%%%%%%%%%%%%%%%%%%%%%%%%%%%%%%%%%%%%%%%%%%%%%%%%%%%%%%
\begin{lemma}\label{LA.3}
Suppose that $\lambda_D>0$. Then for any $\bm f\lneq 0$ in $D\times\cS$,
there exists a unique positive solution $\bm\varphi$ satisfying
$\bcA \bm\varphi = \bm f$ in $D\times\cS$, with
$\bm\varphi=0$ on $\partial D\times \cS$.
\end{lemma}

\begin{proof}
The proof is quite standard and uses the refined maximum principle. 
The latter follows from
the following characterization of $\lambda_D$.
Let $F(D\times\cS)$ denote the collection of all
functions in $\cC(\Bar D\times\cS)\cap\Sobl^{2,d}(D\times\cS)$
which have non-positive values on $\partial D\times\cS$
and are positive at some point in $D\times\cS$.
Then
\begin{equation}\label{ELA.3A}
\lambda_D \,=\, \inf\,\bigl\{\lambda\in\RR \,\colon \exists\, \bPsi\in F(D\times\cS)
\text{\ such that\ }
\bcA \bPsi + \lambda \bPsi \,\ge \,0 \text{\ in\ }  D\times\cS\bigr\}\,.
\end{equation}
Let $\lambda'$ denote the right hand side of \cref{ELA.3A}. 
It follows from \cref{TA.1}
that $\lambda'\le\lambda_D$. Now suppose that for some $\bPsi\in F(D\times \cS)$ we have
\begin{equation*}
\bcA\, \bPsi \,\ge \,-\lambda\,\bPsi
\end{equation*}
in $D\times\cS$ with $\lambda < \lambda_D$.
Since $\bm\varphi_{D}$ in \cref{TA.1} is a super-solution
to $\bcA + \lambda$, repeating a similar argument as in \cref{TA.1},
we see that $\bm\varphi_{D} = t\bPsi$ for some $t>0$ which contradicts
the fact that
$\lambda < \lambda_D$. Thus $\lambda' = \lambda_D$.

As a consequence of the above characterization we have a maximum principle
which can be stated as follows:
if $\bm u$ is a solution to $\cA \bm u\ge 0$ and $\bm u\le 0$ on $\partial D\times \cS$,
then  $\bm u\le 0$ in $D\times\cS$.
Now it is standard to apply a monotone iteration to find a solution
$\bm\varphi$ as stated in the lemma.
\end{proof}

%%%%%%%%%%%%%%%%%%%%%%%%%%%%%%%%%%%%%%%%%%%%%%%%%%%%%%%%%%%%%%%%%%%%%%%%%%%%%%%%%%%%
\subsection*{Acknowledgement}
The authors would like to thank Anindya Goswami for the helpful discussions.
The research of Ari Arapostathis was supported
in part by the National Science Foundation through grant DMS-1715210,
in part by the Army Research Office through grant W911NF-17-1-001,
and in part by Office of Naval Research through grant N00014-16-1-2956
and was approved for public release under DCN\# 43-7339-20.
The research of Anup Biswas was supported in part by a SwarnaJayanti fellowship and DST-SERB grants EMR/2016/004810, 
MTR/2018/000028.



\end{document}